\documentclass{article}

\title{Primitive Positive Constructions Among Finite Permutation Groups}
\author{Sebastian Meyer\thanks{The author has received funding from the European Research Council (Project POCOCOP, ERC Synergy Grant
		101071674). Views and opinions expressed are however those of the author only and do not necessarily reflect those
		of the European Union or the European Research Council Executive Agency. Neither the European Union nor the
		granting authority can be held responsible for them.}}
\date{\today}

\usepackage{amsmath}
\usepackage{amsthm}
\usepackage{amssymb}
\usepackage{leftindex}
\usepackage[utf8]{inputenc}
\usepackage{enumitem}
\usepackage[UKenglish]{babel}
\usepackage{centernot}	% used for --/--> (centered \not)
\usepackage{tikz}
\usepackage[colorlinks=true,linkcolor=blue]{hyperref}

\theoremstyle{plain}
\newtheorem{theorem}{Theorem}[section]
\newtheorem{lemma}[theorem]{Lemma}

\newtheorem{corollary}[theorem]{Corollary}
\newtheorem{conjecture}[theorem]{Conjecture}

\theoremstyle{definition}
\newtheorem{definition}[theorem]{Definition}
\newtheorem{remark}[theorem]{Remark}
\newtheorem*{remark*}{Remark}
\newtheorem{observation}[theorem]{Observation}
\newtheorem{question}[theorem]{Problem}

\DeclareMathOperator{\T}{{\mathbb T}}

\DeclareMathOperator{\Hom}{Hom}
\DeclareMathOperator{\Pol}{Pol}

\theoremstyle{definition}
\newtheorem{examplex}[theorem]{Example}
\newenvironment{example}
	{\pushQED{\qed}\examplex}
	{\popQED\endexamplex}

\DeclareMathOperator{\id}{id}

\newcommand{\de}[1]{\emph{#1}}

\newcommand{\clone}{\mathcal C}
\newcommand{\cloneD}{\mathcal D}

\newcommand{\Sym}{\operatorname{Sym}}
\newcommand{\act}{\curvearrowright} %group action arrow
\newcommand{\prim}{\mathbb{P}}   %the union of the primitive G-sets
\newcommand{\iso}{\cong}    %isomorphism
\newcommand{\stab}{\operatorname{Stab}} %Stabilizer
\newcommand{\normalsubgroup}{\trianglelefteq}   % normal subgroup
\newcommand{\leftfrac}[2]{#2  \backslash #1}	% Right quotient set
\newcommand{\frat}{\Phi}    %Frattini subgroup
\newcommand{\fix}{\operatorname{fix}}

\newcommand{\IZ}{\mathbb{Z}}    %integers
\newcommand{\IN}{\mathbb{N}}    %naturals

\newcommand{\structure}{\mathbb{S}}
\newcommand{\structA}{\mathbb{A}}
\newcommand{\structB}{\mathbb{B}}

\newcommand{\ppchar}{\chi_{\textup{pp}}}	%Charakterisierende Menge für pp-konstruktionen
\newcommand{\condchar}{\chi_{\condition}}	%Charakterisierende Menge für H1-Gleichungen
\newcommand{\condcharred}{\chi_{\condition}^r}	%Charakterisierende Menge für H1-Gleichungen
\newcommand{\condition}{\Sigma}

\newcommand{\minoraction}[2]{\leftindex_{{#1}}{{#2}}}

\newcommand{\ppgle}{\overset{\textup{ppg}}{\longrightarrow}}
\newcommand{\ppgeq}{\overset{\textup{ppg}}{\longleftrightarrow}}
\newcommand{\ecle}{\overset{\textup{e.} \condition}{\longrightarrow}}
\newcommand{\eceq}{\overset{\textup{e.} \condition}{\longleftrightarrow}}
\newcommand{\notecle}{\centernot{\overset{\textup{e.} \condition}{\longrightarrow}}}
\newcommand{\frattiniaction}{Frattini action}

\renewcommand{\subset}{\subseteq}

\begin{document}

\maketitle

\begin{abstract}
    Primitive positive constructions of first order structures and minor conditions have been shown to be a very useful tool in universal algebra for the study of constraint satisfaction problems. 
    However, they seemed to be very rarely studied in classical algebra such as group theory. 
    This paper fills in this gap by considering finite structures and minor conditions induced by permutation groups.
    In this setting, we provide a translation of primitive positive constructions and of implications of minor conditions into group theoretic terms. Moreover, we give a full classification, by linking both partial orders to the epimorphism poset of finite groups with trivial Frattini subgroup.
\end{abstract}

\section{Introduction}

Many problems in theoretical computer science, such as HORN-SAT, $3$-SAT, 3-colorability of graphs, directed reachability in directed graphs, reachability in undirected graphs, may be formulated as a finite constraint satisfaction problem (CSP). 
In this formulation, one has to distribute a certain amount of variables onto a fixed finite template such that some conditions are satisfied respectively return that this is impossible.
Depending on the shape of the template, this computational problem can be in different complexity classes. 
The above examples are $NP$-complete for 3-coloring of graphs and $3$-SAT problems, the HORN-SAT problem is $P$ complete, directed reachability is $NL$-complete and undirected reachability can be solved in $L$, on a logarithmic space. 
So an obvious question is, given any CSP, determine its complexity class. 
Primitive positive constructions have been proven to be a very fruitful tool to this task as on one side a primitive positive construction gives a log-space, and thus polynomial time, reduction of the related problems. On the other side, the existence of a primitive positive construction is equivalent to a purely algebraic condition and thus can be studied completely within the framework of universal algebra. 
In this setting, every template becomes a first-order structure and the existence of a primitive positive construction is completely described in terms of the polymorphism minion, a generalization of the endomorphism monoid. Note that the existence of a primitive positive construction is transitive and thus it defines a preorder on all finite structures.

A general classification of all structures up to primitive positive constructions is not known and probably similar complicated as the classification of all finite simple groups. In fact we have so far some partial classifications as the classification on undirected graphs \cite{HellNesetril}, 2-element structures \cite{vucajBodirskytwoElementPPPoset} and smooth digraphs \cite{smooth-digraphs}. 

This paper considers itself as part of the project to classify all finite structures up to primitive positive constructions. 
In this problem, group-actions or, historically correct, permutation groups appear in two contexts: As a structure within the primitive positive constructability poset and moreover as a minor condition. The second one can be considered within the poset of implications. 
We give a full classification in both contexts, linking both posets to the epimorphism poset of finite groups with a trivial Frattini subgroup.

Permutation group are of a special interest for multiple reasons. When studying CSPs it is of a wide interest to study for a given algorithm the set of structure templates whose CSP can be solved by the algorithm. For multiple algorithms, this set can be described by permutation groups:
\begin{itemize}
    \item A CSP can be solved by basic linear programming robustly in polynomial time if and only if the corresponding structure does not construct any permutation group \cite[Theorem 2]{BLP}, see also Theorem~\ref{TheoremBLPBySimpleGroups}.
    \item A CSP can be solved with the Zhuk Algorithm in polynomial time if and only if the corresponding structure admits a Taylor polymorphism \cite{ZhukFVConjecture,ZhukJournal}. (It is NP complete otherwise.)
    By the cyclic terms theorem \cite[Theorem 4.2]{Cyclic}, this is equivalent to the existence of a permutation group which is not primitively positively constructed.
\end{itemize}
These applications are usually phrases in terms of minor conditions. Getting to know which minor conditions correspond to which structures and which of them pp-construct each other is really helpful.

Moreover, there is the recent result by Meyer and Starke \cite{meyerStarke2024finitesimplegroups} analyzing the full poset of all finite structures ordered by pp-constructions which gives a classification of the simplest non-trivial layer of this poset. It consists of equivalence classes of structures, where representations are given by $\T_3$, a graph on three vertices, and for each finite simple group $G$ a specific $G$-set called $\prim(G)$. So studying all $G$-sets for all groups $G$ is the next step to take. We give in this paper a new explanation why among all groups exactly the simple groups appear: These groups are the topmost groups in the poset of all finite groups ordered by epimorphisms.

In more details, we show in this paper the following classification:
\begin{theorem}
	It is possible to assign to every finite group action $G\act X$ two classes of finite groups with trivial Frattini subgroups $\ppchar(G\act X)$ and $\condchar(G\act X)$ such that
	\begin{itemize}
		\item $\ppchar(G\act X)$ contains finitely many groups up to isomorphism and is closed under group epimorphisms. For every such set, there exists a group action that is characterized by this set.
		\item $\condchar(G\act X)$ contains a finite sets $\condcharred(G\act X)$ of groups and all preimages of it with respect to group epimorphisms (from a group with trivial Frattini subgroup). For every such set, there exists a group action that is characterized by this set.
		\item The structure associated to the group action $G\act X$ pp-constructs the structure associated to a group action $G'\act X'$ if and only if $\ppchar(G'\act X') \subseteq\ppchar(G\act X)$.
		\item The condition associated to the group action $G\act X$ implies the condition associated to a group action $G'\act X'$ if and only if $\condchar(G'\act X') \subseteq\condchar(G\act X)$.
		\item The structure associated to the group action $G\act X$ satisfies the condition associated to a group action $G'\act X'$ if and only if $\condchar(G'\act X') \cap\ppchar(G\act X) = \emptyset$.
	\end{itemize}
\end{theorem}
For more details on the classification, we refer to Sections~\ref{SectionClassificationsGeneral} and~\ref{SectionApplyUniversalAlgebra}.

To enhance readability, the paper is split in two parts. The group theoretic part consisting of Sections~\ref{SectionPreliminariesGroups}, \ref{SectionPPGdefinieren} and \ref{SectionClassification} is self-contained in the sense that it does not assumes any knowledge about universal algebra. Every definition and theorem is completely described in terms of group theory.
In the universal algebra part, consisting of Sections~\ref{SectionPreliminariesUniversalAlgebra} and \ref{SectionApplyUniversalAlgebra}, these results are translated into common universal algebraic language and we show that the reintroduced terms defined in the first part coincide with the standard definition.

In Section~\ref{SectionPreliminariesGroups}, we introduce common notions from group theory. In Section~\ref{SectionPPGdefinieren}, we introduce the new notions of a primitive positive group construction, minor condition and elementary implication of a minor condition. In the following Section~\ref{SectionClassification}, we classify these descriptions completely and show in Section~\ref{SectionApplyUniversalAlgebra} that they indeed agree with the more complicated standard notions from universal algebra, given in Section~\ref{SectionPreliminariesUniversalAlgebra}.

Finally, the Appendix~\ref{SectionExamplesConditions} provides examples of permutation groups and their interpretation as minor conditions in terms of the classification.

%\tableofcontents
\section{Preliminaries from group theory}
\label{SectionPreliminariesGroups}

All sets and groups are considered to be finite. We assume the reader to be familiar with the notion of a group.

The notions are similar as in \cite{meyerStarke2024finitesimplegroups} whenever we overlap with this paper.

We write $[n]$ for the set $\{1,2,\dots,n\}=\IN\cap [1,n]$.

\subsection{Group Actions}
A group action of a group $G$ on a set $X$ is a map $\alpha\colon G \times X \to X$ such that $\alpha((gh),x)=\alpha(g,\alpha(h.x))$ and $\alpha(1,x)=x$ for all $g,h\in G$ and $x\in X$. This action is denoted by $G\act_\alpha X$. If $\alpha$ is clear from the context, we also write $G\act X$ for $G\act_\alpha X$ and $g.x$ for $\alpha(g,x)$. A $G$-set is a set together with a $G$-action. Note that $X$ might be empty.

If $G\act X$ and $H\le G$ is a subgroup, we denote the fixed points $\{x\in X\mid \forall h\in H: h.x=x\}$ of $H$ in $X$ with $\fix_X(H)$ and the set of $H$-orbits with $X/H$. The orbit of $x\in X$ is denoted $H.x$.
The \de{(set) stabiliser} of a subset $S\subset X$ with respect to $G$ is denoted by $\stab_G(S)$. If $G$ is clear from the context we also write $\stab(S)$ instead of $\stab_G(S)$. If $S=\{s\}$ we also write $\stab_G(s)$ instead of $\stab_G(\{s\})$.

A homomorphism $f$ of $G$-sets from $G\act X$ to $G\act Y$ is a map from $X$ to $Y$ such that $f(g.x)=g.f(x)$ for all $g\in G$ and $x\in X$. The set of all homomorphism from $X$ to $Y$ is denoted by $\Hom_G(X,Y)$. We call two $G$-sets $X$ and $Y$ homomorphically equivalent, if both $\Hom_G(X,Y)$ and $\Hom_G(Y,X)$ are non-empty.

A \de{sub-quotient} of a group $G$ is a group that is isomorphic to $S/N$, where $S\le G$ is a subgroup and $N\normalsubgroup S$ a normal subgroup.

We will use many common $G$-sets in this article, such as the following ones:
\begin{itemize}
    \item A group $G$ is itself a $G$-set, where the action is left multiplication. This action is called \de{regular action} and will be relevant mostly for the groups $\IZ/n\IZ$ of integers modulo $n$.
    \item For every group $G$ and every set $X$, there is the \de{trivial action} $g.x=x$.
    \item The group $\Sym(X)$ of all permutations on $X$ has the \de{canonical action} on $X$ by permuting. If $X=\{1,2,\dots,n\}=[n]$, we also write $\Sym(n)\act [n]$ for this action.
    \item The action $D_n\act [n]$ is the subaction of $\Sym(n)\act [n]$ that preserves adjacent numbers where $1$ and $n$ are considered to be adjacent. Alternatively, the action $D_n\act [n]$ is the automorphism group of the regular $n$-gon where the $n$ corners are identified with the set $[n]$. So note that in our definition, $D_n$ has $2n$ elements. (There are two different standards!)
    \item The action $A_n\act [n]$ is the subaction of $\Sym(n)\act [n]$ that consists only of even permutations.
    \item The action $F\act [5]$ is the action of the subgroup of $S_5$ that is generated by the 5-cycle $(12345)$ and the 4-cycle $(1)(2354)$. The group $F$ has 20 elements.
\end{itemize}

We will also use many ways to construct new $G$-sets from given ones. 
Those, that will be relevant, are listed below.
\begin{itemize}
    \item If $G\act_\alpha X$ and $f\colon H\to G$ is a group homomorphism. Then, there is an action $H \act X$ given by $(h,x)\mapsto \alpha(f(h),x)$ called the \de{subgroup action}. Note that we do not require $f$ to be injective.
    \item If $G\act_\alpha X$, $f\colon G\to H$ is a surjective group homomorphism with kernel $N$ and $\fix_X(N)=X$, then there is an action $H \act_\beta X$, such that $\beta(f(g),x)=\alpha(g,x)$ for all $g\in G, x\in X$. This action of $H$ is called \de{quotient group action}.
    \item More general, if $G\act_\alpha X$, $f\colon G\to H$ is a surjective group homomorphism with kernel $N$, then there is an action $H \act_\beta \fix_X(N)$ on the fixed-points of $N$, such that $\beta(f(g),x)=\alpha(g,x)$ for all $g\in G, x\in X$. This action of $H$ is called \de{quotient group sub-action}.
    \item In a similar fashion, if $G\act_\alpha X$, $f\colon G\to H$ is a surjective group homomorphism with kernel $N$, then there is an action $\beta$ of $H$ on the $N$-orbits $X/N$, such that $\beta(f(g),N.x)=N.\alpha(g,x)$ for all $g\in G, x\in X$. This action of $H$ is called \de{quotient group quotient-action}.
    \item The \de{internal union action} of the $G$-actions $G\act_\alpha X$ and $G\act_\beta Y$ is the $G$-action on the disjoint union $X\sqcup Y$ given by $g.x=\alpha(g,x)$ for all $g\in G, x\in X$ and similar $g.y=\beta(g,y)$ for $y\in Y$. It has the universal property of the disjoint union.
    \item The \de{internal product action} of the $G$-actions $G\act X$ and $G\act Y$ is the $G$-action on the set $X\times Y$ given by $g.(x,y)=(g.x,g.y)$ for all $g\in G, x\in X,y\in Y$. It has the universal property of the product.
    \item The \de{power action} of the $G$-actions $G\act X$ to the exponent $n$ is the internal product action $G\act X^n$. For $n=0$, it is defined as the trivial action $G\act \{1\}$.
    \item The \de{external product action} of the actions $\{G_i \act X_i\mid i \in I\}$ is the $\prod_{i\in I}G_i$-action on the set $\prod_{i\in I} X_i$ given by $(g_i)_{i\in I}.(x_i)_{i\in I}=(g_i.x_i)_{i\in I}$ for all $g_i\in G_i, x_i\in X_i$.
    \item For a $G$-set $X$ and any set $Y$, the set $X^Y$ of all maps from $Y$ to $X$ is a $G$-set given by post-composition (or componentwise action), i.e.
    \begin{align*}
        G \times X^Y &\to X^Y \\
        (g,t)&\mapsto g.t\coloneqq (Y\to X, y\mapsto g(t(y))).
    \end{align*}
    \item For a $G$-set $Y$ and any set $X$, the set $X^Y$ of all maps from $Y$ to $X$ is a $G$-set given by pre-composition, i.e.
    \begin{align*}
        G \times X^Y &\to X^Y \\
        (g,t)&\mapsto t_g\coloneqq (Y\to X, y\mapsto t(g.y)).
    \end{align*}
    Note that this is a right action, i.e. the composition rule is replaced by $(gh).t=h.(g.t)$ respectively $t_{gh}=(t_g)_h$.
    \item For $G\act X$ and $H\act Y$, the set of homomorphisms $\Hom_H(Y^X ,Y)$ is a $G$-set given by
    \begin{align*}
        G \times \Hom_H(Y^X ,Y) &\to \Hom_H(Y^X ,Y) \\
        (g,f) &\mapsto \minoraction{g}{f}\coloneqq(t \mapsto f(t_g))
    \end{align*}
    This action is called \de{minor action} and the set $\Hom_H(Y^X ,Y)$ is also denoted $\Pol^{X}_H(Y)$.
\end{itemize}

Recall the definition of the outer semi-direct product or wreath product of a group power and a symmetric group: For a group $G$ and a finite set $X$, the wreath product $\Sym(X) \ltimes G^X$ is given by a pair $(h,f)$ where $f\colon X \to G$ is any map and $h\in \Sym(X)$ is a group element. This is a group where multiplication is given by $(h,f)\cdot(h',f')=(hh',f_{h'}\cdot f')$ where the multiplication of the maps to $G$ is given component wise in $G$. 
This definition allows us to define three more group actions:
\begin{itemize}
    \item For a $G$-set $X$ and a set $Y$, the set $X^Y$ admits an action from the group $\Sym(X) \ltimes G^X$ given by $(h,f).t=(f.t)_{h^{-1}}$, where $f.t$ is applied pointwise. We call this action the \de{full power action}. The $0$-th full power action or full power action with $Y=\emptyset$ is defined as the trivial action $\{1\}\act \{1\}$.
    \item For a $G$-set $X$ and an $H$-set $Y$, the set $X\times Y$ admits an action from the group $G \ltimes H^X$ given by $(g,f).(x,y)=(g.x,f_x.y)$, where $f=(f_x)_{x\in X}$ is the tuple in $H^X$. We call this action the \de{star product action} in reference to the star product in universal algebra, see for example \cite[Section~3.1]{absorption}.
    \item For a $G$-set $X$ and a positive integer $n$, we define the \de{star power action} as the action of the group $G \ltimes G^X \ltimes G^{X^2} \ltimes \dots \ltimes G^{X^{n-1}}$ on $X^n$, which is obtained by repeating the star product.
\end{itemize}

\subsection{The Frattini action}
It will turn out that for each group, there is one very interesting group action given by the disjoint union of all primitive group actions. 

\subsubsection{Maximal subgroups and primitive actions}
We explain primitive group actions in this section.
The following definition and the two observations are taken from \cite[Section 2.1]{meyerStarke2024finitesimplegroups}.

\begin{definition}
    A group action $G \act X$ is called \de{primitive} if $X$ has more than one element and the only  partitions of $X$ that are respected by the $G$ action are $\{X\}$ (a single partition) and $\{\{x\}\mid x \in X\}$ (the discrete partition). A $G$-set is \de{primitive} if its action is.
\end{definition}

\begin{observation}
    For the group action $G\act G/H$ on the left-$H$-cosets, we have $\stab_G(\{H\})=H$. Conversely, if $G\act X$ is a group action with a single orbit and $x\in X$, then $X$ is isomorphic to $G/\stab_G(x)$ as $G$-set, where an isomorphism is given by $G/\stab_G({x})\to X, g\stab_G({x})\mapsto g.x$. This map is well-defined.

    For different points $x$ and $g.x$ in the same orbit of a $G$-action $G\act X$, we have
    $$
        \stab_G(g.x)=g\stab_G(x)g^{-1},
    $$
    so the stabilizers are conjugated. Similarly, the $G$-sets $G/H$ and $G/H'$ are isomorphic if and only if $H$ and $H'$ are conjugated.
\end{observation}

\begin{observation}
It is well known that primitive actions are closely related to maximal subgroups. The following are equivalent for a $G$-action $G\act X$:
\begin{enumerate}
    \item $G\act X$ is primitive.
    \item $G\act X$ is isomorphic to $G/M$ as $G$-set for a maximal subgroup $M<G$.
    \item $G\act X$ is transitive and there exists $x\in X$, such that $\stab_G(x)$ is a maximal subgroup of $G$.
    \item $G\act X$ is transitive and for all $x\in X$, $\stab_G(x)$ is a maximal subgroup of $G$.
\end{enumerate}
The $G$-sets $G/M$ and $G/M'$ are isomorphic if and only if $M$ and $M'$ are conjugated. See for example \cite[Corollary 1.4A and 1.5A]{DixonMortimer}.
\end{observation}

\subsubsection{The Frattini subgroup and minimal fixed-point free actions}
\label{SectionFrattini}
We explain the Frattini action action and its close link to the Frattini subgroup in this subsection.

The \de{Frattini subgroup} $\frat(G)$ of a group $G$ is the intersection 
$$\bigcap_{M\le G \text{ maximal subgroup}} M$$
of all its maximal subgroups. The Frattini subgroup is always a normal subgroup as the conjugations of a maximal subgroups are maximal.

We call a $G$-action on a set $X$ \de{minimal fixed-point free}, if it has no fixed point, but every subgroup action $H\act X$ of a proper subgroup $H\le G$ has a fixed point.
Let $\prim(G)$ for $G\ne \{1\}$ be the internal union of all primitive $G$-actions. We call the $G$ action $G\act \prim(G)$ \de{\frattiniaction{}}. The set $\prim(\{1\})$ is the empty set as we will discuss in Section~\ref{SectionEmptyAction}.
\begin{lemma}[{Lemma 4.6 from \cite{meyerStarke2024finitesimplegroups}}] \label{LemmaMinimalFixedPointFreeHomEquiv}
    Let $G$ be a finite group of order at least two acting on a set $X$. Then $\S(G\act X)$ is homomorphically equivalent to $\S(G\act \prim(G))$ if and only if the action of $G$ on $X$ is a minimal fixed point free action.
    
    In particular, the action of $G$ on $\prim(G)$ is a minimal fixed point free action.
\end{lemma}

Moreover, $G\act \prim(G)$ is a group action, such that an element $g\in G$ acts trivial if and only if it is in the Frattini subgroup $\frat(G)$, as it needs to be in all point stabilizer goups, which are exactly the maximal subgroups.

If $N\normalsubgroup G$ is a normal subgroup then the maximal subgroups of $G/N$ are exactly the quotients $M/N$ where $M\le G$ is a maximal subgroup that contains $N$. This follows immediately from the fact that subgroups of $G/N$ are given by subgroups of $G$ that contain $N$.
Moreover, the $G/N$-group action $\prim(G/N)$ and the quotient group sub-action on $\fix_{\prim(G)}(N)$ are isomorphic.

The quotient $G/\frat(G)$ has a trivial Frattini subgroup as the maximal subgroups of $G/\frat(G)$ are exactly the quotients of the maximal subgroups of $G$.
As $\fix_{\prim(G)}(\frat(G))$ is all of $\prim(G)$, the quotient group action $G/\frat(G) \act \prim(G)$ is well-defined and isomorphic to the \frattiniaction{} $G/\frat(G) \act \prim(G/\frat(G))$.

\subsection{The empty action}
\label{SectionEmptyAction}

We allow in addition to non-empty $G$-sets also the \de{empty action} $G\act \emptyset$, the action of any group on an empty set. This decision differs from \cite{meyerStarke2024finitesimplegroups} but simplifies the theory. As $\{1\}$ has no primitive action, we define $\prim(\{1\})$ to be $\emptyset$. Note that this definition extends Lemma~\ref{LemmaMinimalFixedPointFreeHomEquiv} to the trivial group. Recall that we defined the $G$-set $\emptyset^0$ to be a single point with the trivial action.

\subsection{Biactions}
On the set $X^Y$ of all maps from a $G$-set $X$ to an $H$-set $Y$, we have both a (left) $G$ and a (right) $H$-action. The interplay of the actions gives a possibility to link a subquotient of $G$ with a subquotient of $H$.

\begin{lemma}[{\cite[Lemma~2.4]{meyerStarke2024finitesimplegroups}}]
\label{LemmaIsoOfSubquotient}
    Let $G\act X$ and $H\act Y$ be group actions and let $t\in X^Y$. Let $Z_t=G.t\cap H.t\subseteq X^Y$.
    Then, the maps
    \begin{align*}
        \stab_G(H.t)/\stab_G(t) &\to Z_t \text{ and} & \leftfrac{\stab_H(G.t) }{\stab_H(t)}&\to Z_t\\
        g\stab_G(t) &\mapsto g.t & \stab_H(t) h &\mapsto t_h
    \end{align*}
    are bijections. 
    Moreover, $\stab_G(t)\normalsubgroup \stab_G(H.t)\le G$ and also $\stab_H(t)\normalsubgroup \stab_H(G.t)\le H$  
    and the above maps induce an isomorphism
    $$
        \stab_G(H(t)) / \stab_G(t) \iso \stab_H(G(t)) /\stab_H(t)
    $$
    of groups.
\end{lemma}

We refer the reader to \cite[Lemma~2.4]{meyerStarke2024finitesimplegroups} for the proof.

\section{Preliminaries from model theory and universal algebra}
\label{SectionPreliminariesUniversalAlgebra}

We recall the notions from universal algebra. They will not be used until Section~\ref{SectionApplyUniversalAlgebra}, in which we show that our simplified definitions, which we will define in Section~\ref{SectionPPGdefinieren} specifically for permutation groups, coincide with the standard definitions.

\subsection{Preliminaries on primitive positive constructions}
We refer to \cite{HodgesShort} for the definition of a (fist-order) structure, as well as of products and homomorphisms of them. We assume for simplicity that the structure is relational as every function symbol can be replaced by a relational symbol and we allow the base set to be empty. A primitive positive definition (pp-definition) over a $\sigma$-structure $\structA$ is a first-order definition which is build recursively from
\begin{enumerate}
    \item atomic positive relations, that are
    \begin{itemize}
        \item relations from $\sigma$, including $=$ between arbitrary variables,
        \item true (but not false),
    \end{itemize}
    \item $\exists x: \phi$ whenever $\phi$ is already pp-defined and $x$ is a variable,
    \item $\phi \land \psi$ whenever $\phi$ and $\psi$ are already pp-defined.
\end{enumerate}
Such definition with free variable set $V$ naturally defines a subset of $\structA^V$ by evaluating the formula. We write $x\models \phi$ if $x\in \structA^V$ is inside the set defined by $\phi$. We do not distinguish the structure and its underlying set in our notation.

It is also common to allow false as atomic relation. We do not follow this convention and consider nullary polymorphisms instead to obtain a clean theory.

We say that a $\sigma$-structure $\structA$ \de{primitively positively constructs} or \de{pp-constructs} a $\tau$-structure $\structB$, if there exists a non-negative integer $n$, a primitive postitive formula $\phi$ with $n$ free variables and for each $k$-ary relation symbol $R$ in $\tau$ a pp-formula $\phi_R$ in $kn$ free variables 
such that the $\tau$-structure with base set $\{\Vec{x}\in \structA^n \mid \Vec{x}\models \phi\}$, where the points $(\Vec{x}_1,\dots, \Vec{x}_k)$ are in a $k$-ary relation $R$ if and only if
$$
    (x_{1,1},\dots, x_{n,1},x_{1,2},\dots,x_{n,k-1},x_{1,k},\dots, x_{n,k}) \models \phi_R,
$$
is homomorphically equivalent to $\structB$. Note that the formula $\phi$ can always be chosen to be true, except for the case where $\structB$ is empty.

For a $\sigma$-structure $\structA$ and a $\tau$-structure $\structB$, we define their external direct product as the $\sigma \sqcup \tau$-structure wich has as base set the product of the base sets of $\structA$ and $\structB$ and as relations the full preimages of the projections. A structure pp-constructs an external product if and only if it pp-constructs every factor.

\subsection{Preliminaries on minors}
A \de{minion} $\clone$ consists for each finite set $N$ of a set $\clone_N$ of so-called $N$-ary or $|N|$-ary elements. We sometimes identify $\clone$ with the union $\bigcup_{n\in \IN \cup \{0\}} \clone_n$. Moreover, for each map $\alpha\colon N \to M$, there is a map $\clone_N \to \clone_M, f\mapsto \minoraction{\alpha}{f}$ such that 
\begin{align*}
    \forall f: &&\minoraction{\id}{f}&=f
    \\
    \forall f, \forall \alpha, \beta:&&\minoraction{(\alpha\beta)}{f}&=\minoraction{\alpha}{(\!}\minoraction{\beta}{f})
\end{align*}
hold whenever the arities match. Note that this implies that if $\alpha\colon N \to M$ is a bijection, then the induced map $\clone_N \to \clone_M$ is also a bijection. Equivalently, a minion can be defined as a functor from finite sets to sets.

It is also common to exclude $0$ from the arities. However, including $0$ massively simplifies the consideration of the empty action $\{1\}\act \prim(\{1\})$ later. See also \cite{Nullary} for minions with nullary elements. Note that the set of nullary elements might be empty.

A minor preserving map or \de{minion homomorphism} is a map $F\colon \clone \to \cloneD$ between minions that preserves arity and for each map $\alpha\colon N \to M$, we get $\minoraction{\alpha}{(Ff)}=F(\minoraction{\alpha}{f})$.

A \de{clone} is a minion that has in addition a composition map i.e. an operad structure which is compatible with the minion structure.
\subsection{Minor conditions}
A \de{minor condition} is a property of a minion. Precisely speaking, a minor condition 
consists of two natural numbers $n$ and $m$ and sets $N(1),\dots,N(n)$ and maps $i,j\colon [m] \to [n]$ as well as maps $\alpha_k\colon N(i(k))\to N(j(k))$ for all $k\in [m]$.

We say that a minion $\clone$ satisfies a minor condition $\mathcal{M}$, written $\clone \models \mathcal{M}$, if $\exists f_1\in \clone_{N(1)},\dots, \exists f_n\in \clone_{N(n)}$ and
    $$
        \forall k\in [m] : \minoraction{\alpha_k}{f_{i(k)}}=f_{j(k)}.
    $$

It is easy to see that if $\clone$ and $\cloneD$ are minions, $\clone$ satisfies a minor condition and there is a minion homomorphism $\clone \to \cloneD$, then also $\cloneD$ satisfies the condition.

\begin{example}
    For a group $G$ acting on a set $X$, there is the minor condition for a minion $\clone$ described as
    \begin{align*}
        \exists f \in \clone_X : 
        \forall g \in G: \minoraction{g}{f}=f
    \end{align*}
    where $\minoraction{g}{f}((x_i)_{i\in X})$ is defined as $f((x_{g.i})_{i \in X})$.
    we denote this condition by $\condition(G\act X)$.
\end{example}

\subsection{The polymorphism minion}
For a relational $\sigma$-structure $\structA$, we define the \de{polymorphism minion} of $\structA$, $\Pol(\structA)$, such that $\Pol(\structA)_N$ is the set of all $\sigma$-homomorphisms from the product $\structA^N$ to $\structA$. Recall that $N$ can be any finite set and $\structA^N$ is isomorphic to $\structA^{|N|}$. For each map $\alpha\colon N \to M$, the map $\Pol(\structA)_N \to \Pol(\structA)_M, f\mapsto \minoraction{\alpha}{f}$, is defined by
$$
    \minoraction{\alpha}{f}(x_1,\dots,x_{|M|})=f(x_{\alpha(1)},\dots,x_{\alpha(|N|)})
$$
where we identify $N$ with $|N|$ and $M$ with $|M|$ to be able to write $N$-tuples and $M$-tuples. The structure $\structB^\emptyset=\structB^0$ consists of a single point with all relations.

Nullary polymorphisms and constant unary polymorphisms are closely related and divide the structures into three classes:
\begin{itemize}
        \item If a structure $\structA$ has a nullary polymorphism, then it also has a constant polymorphism and moreover, for every element $a\in \structA$, there is a constant polymorphism onto $a$ if and only if there is a nullary polymorphism onto $a$.
        Such structures exist, such as sets over an empty signature.
        \item There are structures without a nullary polymorphisms, that have no constant polymorphism such as a graph without any edge. For such a graph $G$, its $0$-th power is the graph with a single point and a loop, which cannot be mapped to $G$.
        \item If a structure has no constant polymorphism, then it has no nullary polymorphism by the first bullet. A graph that has an edge, but no loop is an example of such a structure.
\end{itemize}

\begin{example}\label{ExampleGSet}
    We can also turn a $G$-set into a relational structure.

    Let $G$ be a group and $X$ a $G$-set. Then, there is a relational structure $\structure(G\act X)$ over the signature $\sigma=\{v_g\mid g\in G\}$, where each relation symbol $v_g$ is binary, the base set is $X$ and the interpretation of $v_g$ is given by $\{(x,g.x)\mid x\in X\}$.

    For two $G$-structures $X$ and $Y$, we get 
    $$\Hom_G(X,Y)=\Hom_\sigma(\structure(G\act X),\structure(G\act Y))$$
    and $(\structure(G\act X))^n$ equals $\structure(G\act X^n)$. Thus, the polymorphisms of $X$ as $G$-set are exactly its polymorphisms as $\sigma$-structure.

    A $G$-set has a nullary polymorphism if and only if it has a fixed point. The empty action has a constant, but no nullary operation. All other $G$-sets have no constant polymorphism.
\end{example}

\begin{theorem}[{\cite[Theorem 1.3]{Wonderland}}]
    \label{TheoremReflections}
    Let $\structA$ and $\structB$ be finite relational structures. Then, the following are equivalent:
    \begin{enumerate}
        \item There is a minion homomorphism $\Pol(\structA)\to \Pol(\structB)$.
        \item The structure $\structA$ primitively positively constructs the structure $\structB$.
        \item The minion $\Pol(\structB)$ satisfies all minor conditions satisfied of $\Pol(\structA)$.
    \end{enumerate}
\end{theorem}

The complexity of checking wether a structure satisfies a minor condition related to a group action depends to some extend on this action, see \cite{KK22} for a discussion.

\section{Minor conditions for permutation groups}
\label{SectionPPGdefinieren}

Recall that this article is motivated by the theory of primitive positive constructions and minor conditions in universal algebra. As permutation groups are only a special case, we can simplify these notions and replace all terms from universal algebra by notions from group theory.
We introduce the simplified notions here.

\subsection{Minors and Loop conditions} \label{SectionDefineLoopCond}

In this subsection, we define loop conditions, a special case of minor conditions from universal algebra.

In general, we want to study in this article the poset where a group action is below a second one, if the second one satisfies all loop conditions that the first one satisfies. We show that we can construct the second permutation group in this case from the first one by a construction, which we call primitive positive group construction (ppg-construction). This construction will be defined in Section~\ref{SectionDefinePPGConst} and is similar to the concept of a primitive positive construction in universal algebra.

Moreover, we can also consider the preorder on loop conditions, where one condition implies another if every permutation group satisfying the first condition also satisfies the second one. 
In this case, we can also construct the second loop condition from the first one in a somehow similar construction as the ppg-construction. 
We call this construction elementary implication and define it in Section~\ref{SectionDefineMCI}. 
We are not aware of a similar concept in universal algebra for general minor conditions and assume that if there would be one, it would need to be much more complicated.

\begin{definition} \label{DefinitionMinorForGroups}
    For a groups action $G\act X$ and a finite non-empty set $Y$, we define the \de{$Y$-ary polymorphisms} of $X$ to be the $G$-set homomorphisms $\Hom_G(X^Y,X)$. We denote the $Y$-ary polymorphism with $\Pol^Y_G(X)$ and the class of all polymorphism with $\Pol_G(X)$.

    Recall that if $Y$ is an $H$-set, then $\Pol^Y_G(X)$ becomes an $H$-set by 
    \begin{align*}
        h.f\colon X^Y &\to Y\\
        t &\mapsto f(t_h)
    \end{align*}
    where $t_h$ is the map $X\to Y, x\mapsto t(h(x))$.
    We denote the map $h.f$ by $\minoraction{h}{f}$ and call it the $h$-minor of $f$.

    If $Y$ is an $H$-set, then we say that $\Pol_G(X)$ satisfies the $H\act Y$-\de{condition}, \de{minor condition} or \de{loop condition}, as formula
    $$
        \Pol_G(X) \models \condition(H\act Y),
    $$
    if the action of $H$ on $\Pol^Y_G(X)$ has a fixed point.
\end{definition}
\begin{definition} 
    For $Y=\emptyset$ and a groups action $G\act X$, we define the \de{$Y$-ary polymorphisms} of $X$ as the set of all $G$-fixed points of $X$. So the condition $\condition(\{1\}\act \emptyset)$ is equivalent to the existence of a fixed point.
\end{definition}

In general, we want to describe, in which case $\Pol_G(X) \models \condition(H\act Y)$ holds.
There are already some partial results on this question in \cite{meyerStarke2024finitesimplegroups}, which we want to mention here.
\begin{lemma}[{\cite[Lemma 4.15]{meyerStarke2024finitesimplegroups}}]
    \label{LemmaConditionElementary}
    Let $G\act X$ and $H\act Y$ be finite group actions. Then, the following are equivalent:
    \begin{enumerate}
        \item $\Pol_G(X) \not \models \condition(H\act Y)$
        \item There exists $t\in X^Y$ such that $\stab_G(H.t)\act X$ has no fixed point.
    \end{enumerate}
\end{lemma}

\begin{lemma}[{\cite[Lemma 4.16]{meyerStarke2024finitesimplegroups}}]
    \label{LemmaNoSelfLoop}
    Let $G\act X$ be a finite group action without a fixed point. Then, 
    $\Pol_G(X)\not\models \condition(G\act X)$.
\end{lemma}

\begin{lemma}[{\cite[Theorem A.4]{meyerStarke2024finitesimplegroups}}] \label{LemmaSymCondition}
    Let $G$ be a group.
    The condition $$\Pol_G(\prim(G))\models \condition(\Sym(n)\act[n])$$ holds if and only if $n$ cannot be written as a sum, where each summand is a non-negative multiple of the size of a connected component of $X$.    
\end{lemma}

\begin{example} We can easily get two examples from Lemma~\ref{LemmaSymCondition}.
    \begin{itemize}
        \item For the regular representation $\IZ/2\IZ \act \IZ/2\IZ$, we have $\Pol_{\IZ/2\IZ}(\IZ/2\IZ)\models \condition(\Sym(n)\act[n])$ if and only if $n$ is odd.
        \item The set $X=\prim(\IZ/10\IZ)$ is the disjoint union $\IZ/2\IZ \cup \IZ/5\IZ$ in the following sense: Let $X$ be the set $\{1,2,3,4,5,6,7\}$ with the group action of $\IZ/10\IZ$ such that $1\in \IZ/10\IZ$ maps 1 to 2 to 1 and 3 to 4 to 5 to 6 to 7 to 3. Then, 
        $$ 
            \Pol_{\IZ/10\IZ}(X)\not\models \condition(\Sym(n)\act[n])
        $$
        for all $n$, except $n=1$ and $n=3$. \qedhere
    \end{itemize}
\end{example}

Now, we can make the introduction a formal text.
Note that this correspondence in Definition~\ref{DefinitionMinorForGroups} induces two partial orders on group actions: 
\begin{enumerate}
    \item For two group actions $G\act X$ and $G'\act X'$, we have a notion that $G\act X\le G'\act X'$, if for every group action $H\act Y$, we have 
    $$
        \Pol_G(X)\models \condition(H\act Y) \implies \Pol_{G'}(X')\models \condition(H\act Y).
    $$
    \item For two group actions $H\act Y$ and $H'\act Y'$, we have a notion that $H\act Y\le H'\act Y'$, if for every group action $G\act X$, we have 
    $$
        \Pol_G(X)\models \condition(H\act Y) \implies \Pol_{G}(X)\models \condition(H'\act Y').
    $$
\end{enumerate}
We want to give elementary descriptions of these two preorders. 
The first will have an equivalent characterization as a primitive positive group construction and the second as elementary loop condition implication.

\subsection{Primitive positive group-constructions}
\label{SectionDefinePPGConst}
The notion of a primitive positive group-constructions is simplified concept similar to the notion of a primitive positive construction from universal algebra. 

\begin{definition}
    A group action $G\act X$ primitively positively group-constructs (ppg-constructs) a group action $H \act Y$, if starting with the action $G\act X$ and taking
    \begin{enumerate}
        \item a full power action followed by
        \item a subgroup action followed by
        \item a quotient group sub-action
    \end{enumerate}
    can result in a group action $H\act Y'$ of the group $H$, that differs from $H\act Y$ only by
    \begin{enumerate}[resume]
        \item a homomorphic equivalence.
    \end{enumerate}
    
    Two group actions are primitively positively group-interconstructible (ppg-interconstructible), if the ppg-construct each other.

    For a finite set of group actions $\{G_i\act X_i \mid i\in I\}$ we say that it ppg-constructs a group action $H\act Y$ if the external product action $\prod_{i\in I} G_i \act \prod_{i\in I} X_i$ does.
\end{definition}   

\begin{example}
    The regular group action $\IZ/2\IZ \act \IZ/2\IZ$ ppg-constructs the regular group action $\IZ/4\IZ \act \IZ/4\IZ$, as taking the second full power gives the action 
    $$
        \Sym(2) \ltimes(\IZ/2\IZ)^2 \act (\IZ/2\IZ)^2 
    $$
    which is also known as the action of $D_4$ on 4 points. Here, we find $\IZ/4\IZ$ as a subgroup of $D_4$ such that its action on $[4]$ is isomorphic to the regular action of $\IZ/4\IZ$.
\end{example}

\subsection{Minor Condition Implications}
\label{SectionDefineMCI}
The question which minor conditions imply each other can be very hard for general finite domain structures and general minor conditions. However, in the case of conditions coming from group actions there is a practical description. 

\begin{definition}
    A group action $H\act Y$ elementary implies a second group action $H'\act Y'$ as loop-conditions,
    if it can be obtained by the following steps:
    \begin{enumerate}
        \item Replace the group action $H\act Y$ by a star power $(H^{Y^{n-1}}\rtimes \dots \rtimes H^Y\rtimes H)\act Y^n$. 
        \item Then, take a subgroup action $H_2\act Y^n$.
        \item Then, replace the current group action $H_2\act Y^n$ by another group action $H_2\act Y'$ where there exists a $H_2$-set homomorphism $Y^n \to Y'$.
        \item Finally, take $H'\act Y'$ as a quotient group action of $H_2\act Y'$, where the kernel acts trivial.
    \end{enumerate}
    In this case, we write $(H\act Y)\ecle (H'\act Y')$.

    Two group actions are elementary equivalent as conditions, written $(H\act Y)\eceq (H'\act Y')$, if they elementary imply each other.

    We define that a finite set $\{H_i\act Y_i \mid i \in I\}$ of group actions elementary implies as loop-condition a group action $H'\act Y'$, if the external product action $\prod_{i\in I} H_i\act \prod_{i\in I}Y_i$ does. The empty product is defined as the action of the trivial group on a single point.
\end{definition}   

\subsection{Minors with the empty action}
Recall that we allow in addition to non-empty $G$-sets also the empty action $G\act \emptyset$.
We now clarify the previous definitions to this case.

We define for all $G$ that $G\act \emptyset$ and $\{1\}\act \emptyset$ are ppg-interconstructable as the group acts trivial. No subgroup of $G$ has a fixed point in this action, thus the minimal fixed-point free group is $\{1\}$. 
A non-empty action $G\act X$ ppg-constructs $\{1\}\act \emptyset$ if and only if it has no fixed point as in this case, $\emptyset$ can be constructed as the quotient group subaction of $G/G$ on the fixed-points of $G$. The action $\{1\}\act \emptyset$ ppg-constructs $G\act X$ if and only if $X$ has a fixed point: In this case, the $0$-th full power action of the empty action  $\{1\}\act \emptyset$ is the trivial action $\{1\}\act \{1\}$, which is ppg-interconstructable with $G\act X$.

As condition, we define $\condition(G\act \emptyset)$ to be the condition of having a fixed point. That is reasonable as a $0$-ary polymorphism of a $G$-set $X$ should be a map from $X^0$ to $X$.
Therefore, we get for a non-empty $G$-set $X$ that
\begin{align*}
    \Pol_G(X) &\models \condition(\{1\}\act \emptyset) \iff G\act X\text{ has a fixed point},
    \\
    \Pol_G (\emptyset) &\not\models \condition(\{1\}\act \emptyset),
    \\
    \Pol_G (\emptyset) &\models \condition(G\act X).
\end{align*}
where the second line refers to $\emptyset$ not having a fixed p
To understand $\Pol (\emptyset) \models \condition(G\act X)$, note that $\emptyset^X$ is still empty and thus $G$-invariant.

Finally, we get as elementary implications that all conditions of any group on the empty set are equivalent. Moreover, for any non-empty action $G\act X$, we get
\begin{align*}
    \condition(\{1\}\act \emptyset) &\ecle \condition (G \act X),
    \\
    \condition (G \act X) & \notecle \condition(\{1\}\act \emptyset)
\end{align*}
The first follows as $G\act \emptyset$ is a subgroup action and has a $G$-set homomorphism into $G \act X$. The second follows as an empty set cannot be obtained by a star power, a subgroup action, a homomorphism nor a quotient group action.

\begin{remark}
    Note that Lemma~\ref{LemmaConditionElementary}, Lemma~\ref{LemmaNoSelfLoop} and Lemma~\ref{LemmaSymCondition} are also valid for the empty action.
\end{remark}

\section{The classification of group actions up to ppg-constructions}
\label{SectionClassification}

We want to give a classification which group actions can be ppg-construct from another action, potentially of another group, and we want to know in which cases $\condition(G\act X)$ implies $\condition(H\act Y)$.

\subsection{Basic Facts}
We begin with some simple observations.

\subsubsection{ppg-constructions}
We start by proving three basic facts about ppg-constructions. For a more detailed discussion, we refer the reader to \cite{Wonderland}, in which similar lemmata are shown in the setting of pp-constructions.

\begin{lemma} \label{LemmaPPConstructionsPreserveMinorConditions}
    Let $G\act X$ and $H\act Y$ be two group actions, such that $G\act X$ ppg-constructs $H\act Y$. Then, $\Pol_H(Y)$ satisfies all loop conditions of $\Pol_{G}(X)$.
\end{lemma}
\begin{proof}
    This is a straight forward check of conditions: A ppg-construction consists of four steps, each of them clearly respects loop conditions. 
    We omit the details.
\end{proof}
\begin{lemma} \label{LemmaPPConstructionsTransitive}
    Let $G\act X$, $(G'_i\act X'_i)_{i\in I}$ and $G''\act X''$ be group actions where $I$ is a finite set. Assume that $G\act X$ ppg-constructs $G'_i\act X_i$ for all $i\in I$ and $\{G'_i\act X'_i\mid i \in I\}$ ppg-constructs $G''\act X''$. Then also $G\act X$ ppg-constructs $G''\act X''$.\end{lemma}
\begin{proof}
    First, consider the case where $I$ contains only a single element. In this case, it is a straight forward check of conditions: A ppg-construction followed by another ppg-construction is just an application of the four steps of this constructions in a different order and quantity. However, they can always be reordered to this specific order. Moreover, applying the same step twice can always be replaced by a single application of this step. We omit the details.

    For the general case, note that we can replace the set in the middle by the single external product action $G'\act X'$, as $G\act X$ ppg-constructs all of $G'_i\act X'_i$ only if it ppg-constructs their external product action: 
    We know by assumption that there exist group actions 
    $G_{i,1}\act X_{i,1}, G_{i,2}\act X_{i,1}$ and $G'_{i}\act X_{i,3}$ for all $i\in I$
    such that $G_{i,1}\act X_{i,1}$ is a full power of $G\act X$, $G_{i,2}\act X_{i,1}$ is a subgroup action of $G_{i,1}\act X_{i,1}$, $G'_{i}\act X_{i,3}$ is a quotient group sub-action of $G_{i,2}\act X_{i,1}$ and $G'_i\act X'_i$ is homomorph equivalent to $G'_{i}\act X_{i,3}$. Now, clearly 
    $$
        \prod_{i\in I} G_{i,1} \act \prod_{i\in I} X_{i,1}
    $$
    is a full power of $G\act X$ and has 
    $
        \prod_{i\in I} G_{i,2} \act \prod_{i\in I} X_{i,1}
    $
    as subgroup action which moreover has the quotient group sub-action
    $
        \prod_{i\in I} G'_{i} \act \prod_{i\in I} X_{i,3}
    $
    that is homomorphically equivalent to the product
    $$
        \prod_{i\in I} G'_i \act \prod_{i\in I} X'_i
    $$
    as desired. 
    
    Thus, we may apply the case where $I$ contains just a single element.
\end{proof}
\begin{lemma} \label{LemmaPPConstructMinorAction}
    Let $G\act X$ and $H\act Y$ be two group actions. Then, the action $G\act X$ ppg-constructs the minor action $H\act \Pol_G^Y(X)$.
\end{lemma}
\begin{proof}
    We have by definition $G\act X \ppgle \Sym(X^Y) \ltimes G^{X^Y} \act X^{X^Y}$, the full power action. Note that we have on $X^{X^Y}$ an action of $H$ by 
    $$
        \forall h\in H, f\in X^{X^Y}, t\in X^Y:
        (h.f)(t)=f(t_h)  
    $$
    Moreover, we have a conjugation action of $G$ by 
    $$
        \forall g\in G, f\in X^{X^Y}, t\in X^Y: (g.f)(t)=g.(f(g^{-1}.t))
    $$
    and those two actions are independent in the sense that the $H$-action on a $G$-fixed-point is again a $G$-fixed-point as the first action acts only on $Y$ and the second one on the two $X$ of $X^{X^Y}$.
    Both of these actions are subactions of the $\Sym(X^Y) \ltimes G^{X^Y}$-action, thus we get homomorphisms $\alpha\colon G \to \Sym(X^Y) \ltimes G^{X^Y}$ and $\beta\colon H \to \Sym(X^Y) \ltimes G^{X^Y}$ inducing those actions. Now, $\Sym(X^Y) \ltimes G^{X^Y} \act X^{X^Y}$ ppg-constructs $H \act_{\beta} \fix_{\alpha(G)}(X^{X^Y})$ as this is a subgroup action of a quotient group subaction. However, the set $\fix_{\alpha(G)}(X^{X^Y})$ are exactly those elements in $X^{X^Y}$ that are compatible with $G$ in the sense that they are polymorphisms $X^Y\to X$ and the action of $H$ is exactly the minor action.
\end{proof}

\subsubsection{loop conditions}
As for ppg-constructions, we want to give three basic lemmata for elementary implications of loop conditions.

Elementary implications of loop conditions compose in the following sense:
\begin{lemma}
    Let $H\act Y$, ($H'_i\act Y'_i)_{i\in I}, H''\act Y''$ be group actions where $I$ is a finite set.
    Assume that for all $i\in I$, $(H\act Y)\ecle (H'_i\act Y'_i)$ and that $\{H'_i\act Y'_i\mid i\in I\}\ecle (H''\act X'')$. Then, we also get $(H\act Y)\ecle (H''\act X'')$.
\end{lemma}
\begin{proof}
    The proof is similar to Lemma~\ref{LemmaPPConstructionsTransitive}.
\end{proof}

Elementary implications are indeed implications in the sense of loop conditions:
\begin{lemma}   \label{LemmaElImplImplyImply}
    Let $G\act X$, $H\act Y$ and $H'\act Y'$ be group actions such that $\Pol_G(H)\models \condition(H\act Y)$ and $(H\act Y)\ecle(H'\act Y')$. Then, $\Pol_G(H)\models \condition(H'\act Y')$.
\end{lemma}
\begin{proof}
    Let $H_1\act Y^n$ be the $n$-th $*$-power of $H\act Y$, $H_2\to H_1$ a homomorphism, $Y_n\to Y'$ an $H_2$-set homomorphism and $H_2\to H'$ an epimorphism whose kernel acts trivial on $Y'$. 
    These object exist by the definition of an elementary implication.
    Let $f\in \Pol_Y(X)$ be invariant under $H$. Then, the $n$-folded composition 
    $$
        f_1\coloneqq f\circ(f,\dots,f)\circ (f,\dots,f,\dots,f,\dots,f)\circ \dots \colon X^{Y^n} \to X
    $$
    considered as $Y^n$-ary map satisfies the $*$-power condition $\condition(H_1\act Y^n)$. Clearly, $f_1$ also satisfies $\condition(H_2\act Y^n)$, as being invariant under $H_1$ implies being invariant under $H_2$. Now, the $H_2$-set homomorphism $Y^n\to Y'$ induces a $G$-set homomorphism $G^{Y'} \to G^{Y^n}$. Pre-composing $f_1$ with it defines a polymorphism $f_2\colon G^{Y'}\to G$, that is invariant under $H_2\act Y'$. As these are the same conditions on permutations, it is also invariant under $H'\act Y'$.
\end{proof}

\begin{lemma}   \label{LemmaConditionsCompose}
    Let $G\act X$ and $G\act Y$ be two group actions. Then, 
    $$\{G \act Y\}\cup \{H \act X \mid H\le G \text{ stabilizer subgroup of a point }y\in Y\} \ecle G\act X$$
    holds.
\end{lemma}
\begin{proof}
    At first, we simplify and replace the conditions 
    $$
     \{G \act Y\}\cup \{H \act X \mid H\le G \text{ stabilizer subgroup of a point }y\in Y\}
    $$
    by a single condition
    $$
        G \ltimes (\prod_{y \in Y} H_{i(y)} ) \act  Y \times X
    $$
    where $H_y$ is the stabilizer subgroup of $G$ at $y\in Y$,  $i(y)\in Y$ is a representative of the connected component of $y$ and $I\subseteq Y$ will be the subset containing all representatives. 
    
    The action is defined in the following way: The group $H_{i(y)}$ acts trivial on $Y$. In $X$ direction, it acts nontrivial only on those points in $X \times Y$  where the $Y$-coordinate equals $y$. On these points, it acts on $X$ by the subgroup action of $G\act X$. The factor group $G$ acts trivial in $X$-direction and by the given $G\act Y$-action on the $Y$-coordinates. Note that the semi-direct product here is well-defined, as every $G$-orbit of the $G$-action on $Y$ exchanges only isomorphic factors of the product $\prod_{y \in Y} H_{i(y)}$. Note that the $G$-set $Y$ is isomorphic to $\coprod_{i\in I} G/H_i$. 
    
    To get this simplification, note that
    \begin{align}
        &   \label{eqCompose1}
        \{G \act Y\}\cup \{H \act X \mid H\le G \text{ stabilizer subgroup of a point }y\in Y\}
        \\ \ecle{}& \label{eqCompose2}
        G \times \prod_{y\in Y} H_y \act Y \times \prod_{y\in Y} X
        \\ \ecle{}& \label{eqCompose3}
        (G \times \prod_{y\in Y} H) \ltimes (G \times \prod_{y\in Y} H_y)^{Y \times \prod_{y\in Y} X} \act (Y \times \prod_{y\in Y} X) \times (Y \times \prod_{y\in Y} X)
        \\ \ecle{}& \label{eqCompose4a}
        G \ltimes (\prod_{i\in I} H_i)^{Y} \act (Y \times \prod_{y\in Y} X) \times (Y \times \prod_{y\in Y} X)
        \\ \ecle{}& \label{eqCompose5}
        G \ltimes (\prod_{i\in I} H_i)^{Y} \act (Y) \times (\prod_{i\in I} X)
        \\ ={}& \label{eqCompose6}
        G \ltimes (\prod_{y \in Y} \prod_{i\in I} H_i) \act  Y \times (\prod_{i\in I} X) 
        \\ \ecle{}& \label{eqCompose7}
        G \ltimes (\prod_{y \in Y} H_{i(y)} ) \act  Y \times X
        \\ ={}& \label{eqCompose7a}
        G \ltimes (\prod_{y \in Y} H_{i(y)} ) \act  \coprod_{i\in I} G/H_i \times X
    \end{align}
    where all actions are the canonical ones. 
    In more details, we begin with the definition to get action \eqref{eqCompose2}. Taking a star product gives action \eqref{eqCompose3}. Then, we take the subgroup to get \eqref{eqCompose4a}. This action is just a disjoint union of \eqref{eqCompose5} respectively \eqref{eqCompose6}. A similar subgroup action gives a disjoint union of action \eqref{eqCompose7} which we thus get by a homomorphic equivalence and which can be rewritten as \eqref{eqCompose7a}.
    
    In the next step, we want to find a group homomorphism $\alpha'\colon G\to G\ltimes (\prod_{i\in I} H_i^{G/H_i}) $ such that the subgroup action
    \begin{align}
        \label{eqCompose8}
        G \act_\alpha  \coprod_{i\in I} G/H_i \times X
    \end{align}
    is isomorphic to a disjoint union of the $G$-actions on $X$.
    We cannot take the usual embedding of $G$ into the semi-direct product, as this would be precisely the trivial action on $X$. 

    Recall that we identified $Y$ with $\coprod_{i\in I} G/H_i$. For $y\in Y$ choose $J(y)\in G$ as a preimage of $y$ under the composition $G\to G/H_{i(y)} \to Y$.
    For simplicity let $r(g,y)=(J(g.y))^{-1}\cdot g\cdot J(y)\in H_{i(y)}$.

    We are now able to define the action $\alpha$ as
    \begin{align*}
        \alpha \colon G & \act Y \times X 
        \\
        (g,(y,x)) & \mapsto \left( g.y,r(g,y).x\right) 
    \end{align*}
    This action clearly is the same as the subaction inherited from the homomorphism
    \begin{align*}
        \alpha'\colon G & \to G \ltimes (\prod_{y\in Y} H_{i(y)}) \iso G \ltimes (\prod_{i\in I} H_i^{G/H_i}) \\
        g & \mapsto (g,(r(g,y))_{y\in Y}).
    \end{align*}
    This is indeed a homomorphism as it maps $1\in G$ to the unit and
    \begin{align*}
          &\phantom{{}={}}\alpha'(g) \circ \alpha'(h)
        \\&= 
        (g,y\mapsto r(g,y) ) \circ (h,r(h,y))
        \\ &=
        (gh, y\mapsto (r(g,h.y)\cdot r(h,y)))
        \\ &= 
        (gh, y\mapsto (J(gh.y))^{-1}gJ(h.y)(J(h.y))^{-1} h J(y))
        \\ &= 
        (gh, y\mapsto r(gh,y))
        \\&= \alpha'(g\circ h).
    \end{align*}
    So the action \eqref{eqCompose8} can be ppg-constructed. 

    To get a better understanding of the action \eqref{eqCompose8}, consider any $(y,x)\in Y\times X$ and its stabilizer subgroup with respect to $\alpha$. This stabilizer subgroup is exactly given by the set of all elements $g\in H_i\subseteq G$ that are in the stabilizer subgroup of $x$. Thus, there is a $G$-set homomorphism from $Y\times X$ (with the $\alpha$-action) to $X$ which is on this component defined by mapping $(x,y)$ to $x$. Thus,
    \begin{align}
        G \act X
    \end{align}
    is elementary implied by \eqref{eqCompose8}. We want to remark that this homomorphism is not necessarily given by the projection, but we also do not need it to be.
\end{proof}

\subsection{Comparing primitive actions}
At first, we want to consider Frattini actions, that are group actions of the form $G\act \prim(G)$ and give a classification when one of these actions ppg-constructs or elementary implies a second one. 

We start with a Lemma useful for both, ppg-constructions and elementary implications.

\begin{lemma} \label{LemmaPrimitiveClassificationToEpi}
    Let $G$ and $H$ be two groups.
    If
    $$\Pol_G(\prim(G)) \not\models \condition (H \act \prim(H))$$
    then there is a group epimorphism $G \twoheadrightarrow H/\frat(H)$.
\end{lemma}
\begin{proof}
This is a generalization of \cite[Lemma 4.17]{meyerStarke2024finitesimplegroups}.
    By Lemma~\ref{LemmaConditionElementary}, there exists $t\in \prim(G)^{\prim(H)}$ such that $\stab_G(H.t)\act \prim(G)$ has no fixed point. Fix such a $t$. As every proper subgroup of $G$ has a fixed-point in $\prim(G)$, we get that $\stab_G(H.t)=G$.

    Assume that $\stab_H(G.t)$ would be a proper subgroup of $H$. Then, the action $\stab_H(G.t)\act \prim(H)$ would have a fixed point $y$. Now, we would get for all $g\in G$ that there exists $h\in H$ with $g.t=t_h$ as $G=\stab_G(H.t)$. Hence $(g.t)(y)=t_h(y)=t(h.y)=t(y)$. Therefore, $t(y)$ is a fixed-point of $G\act \prim(G)$, a contradiction. Hence, $\stab_H(G.t)=H$.

    Thus, we get by Lemma~\ref{LemmaIsoOfSubquotient} that
    $$
        G/\stab_G(t) = \stab_G(H.t)/\stab_G(t) \iso \stab_H(G.t)/\stab_H(t)=H/\stab_H(t).
    $$
    So it is left to show that $\stab_H(t)=\frat(H)$ to get that $H/\frat(H)$ is isomorphic to a quotient of $G$.
    
    Note that $\frat(H)\le \stab_H(t)$ as the action of $\frat(H)$ on $\prim(H)$ is trivial. For the other inclusion let $M\le H$ be a maximal subgroup and consider any $y\in \prim(H)$ such that $M = \stab_H(y)$. Note that $t_h(y)=t(y)$ for all $h\in M$. Now, the set $S$ of all $g\in G$ such that $(g.t)(y)=t(y)$ is a subgroup of $G$, namely the maximal subgroup $\stab_{G\act X}(t(y))$. Note that $S$ contains $\stab_G(t)$ as every group element that acts trivial on $t$ cannot change $t(y)$. Thus, the set $S/\stab_G(t)$ of all elements $[g]$ in $G/\stab_G(t)$ with $([g].t)(y)=t(y)$ is also a maximal subgroup. Now, we can use Lemma~\ref{LemmaIsoOfSubquotient} and get that the set of all elements in $H/\stab_H(t)$ that preserve $t(y)$ is a maximal subgroup and thus a proper subgroup. Thus, the set of all $h\in H$ with $t_h(y)=t(y)$ has to be a proper subgroup and thus equal to $M$. So $\stab_H(t)\le M$. As $M$ was arbitrary, this holds for all maximal subgroups and thus $\stab_H(t) \le \frat(H)$. It follows that $\stab_H(t) = \frat(H)$ and $G/\stab_G(t)\iso H/\frat(H)$.
\end{proof}

\subsubsection{ppg-constructions}

We can use the previous lemma to describe ppg-constructions between Frattini actions.

\begin{corollary} \label{CorPrimitiveClassificationOnlyIfDirection}
	If the group action $G\act \prim(G)$ ppg-constructs the action $H\act \prim(H)$ then  we have
    $$\Pol_G(\prim(G)) \not\models \condition (H \act \prim(H))$$ and
    there is a group epimorphism $G \twoheadrightarrow H/\frat(H)$.
\end{corollary} 
\begin{proof}
    We have by Lemma~\ref{LemmaNoSelfLoop} that 
    $$
        \Pol_H(\prim(H)) \not\models \condition (H \act \prim(H))
    $$
    and by Lemma~\ref{LemmaPPConstructionsPreserveMinorConditions} also 
    $
        \Pol_G(\prim(G)) \not\models \condition (H \act \prim(H))
    $.
    Thus, we can apply Lemma~\ref{LemmaPrimitiveClassificationToEpi}.
\end{proof}

\begin{lemma} \label{LemmaFrattiniQuotient}
    Let $G$ be a group.
	The group actions $G\act \prim(G)$ and $G/\frat(G) \act \prim(G/\frat(G))$ are ppg-interconstructible.
\end{lemma}
\begin{proof}
	Note that maximal subgroups of $G/\frat(G)$ are exactly the quotients of maximal subgroups of $G$. Thus we can identify $\prim(G)$ with $\prim(G/\frat(G))$. Then, this is the subgroup action respectively the quotient group action with respect to the group homomorphism $G\to G/\frat(G)$.
\end{proof}

\begin{lemma} \label{LemmaPrimitiveClassificationIfDirection}
	The group action $G\act \prim(G)$ ppg-constructs the group action $H\act \prim(H)$ if there is a group epimorphism $G \twoheadrightarrow H/\frat(H)$.
\end{lemma}
\begin{proof}
	If there is such a group epimorphism, then let $N$ be its kernel and note that the epimorphism induces an isomorphism of $H/\frat(H)$ and $G/N$. Note furthermore that the maximal subgroups of $G/N$ are given by quotients of the maximal subgroups of $G$ that contain $N$. Thus,
    \begin{align*}
        G\act \prim(G) 
        \ppgle &&G/N &\act \stab_{\prim(G)}(N)  &\text{(quotient action)}
        \\\iso &&G/N &\act \prim(G/N)&\text{(see Section~\ref{SectionFrattini})}
        \\\ppgeq &&H/\frat(H) &\act \prim(G/N) &\text{(isomorphic groups)}
        \\\iso &&H/\frat(H) &\act \prim(H/\frat(H)) &\text{(isomorphic groups)}
        \\\ppgeq && H &\act \prim(H) &\text{(Lemma~\ref{LemmaFrattiniQuotient})}
    \end{align*}
    as claimed.
\end{proof}

\subsubsection{loop conditions}

We describe for which groups $G$ and $G'$ the relation $G'\act \prim(G')\ecle G\act \prim(G)$ holds.

\begin{lemma} \label{LemmaConditionImpliesEpiShortened} 
    Consider two groups $G$ and $G'$.
    Assume that
    $\Pol_G(\prim(G)) \models \condition(G' \act \prim(G'))$. 
    Then, $G'\act \prim(G')$ cannot elementary imply as condition $G \act \prim(G)$.
\end{lemma}
\begin{proof}
    In this case, we have 
    \begin{align*}
        \Pol_G(\prim(G)) &\not \models \condition(G \act \prim(G))&& \text{by Lemma~\ref{LemmaNoSelfLoop} and}\\
        \Pol_G(\prim(G)) &\models \condition(G' \act \prim(G'))&& \text{by assumption.}
    \end{align*}
     Thus, the second condition cannot elementary imply the first one by Lemma~\ref{LemmaElImplImplyImply}.
\end{proof}

\begin{lemma} \label{LemmaEpiImpliesCondition}
    Consider two groups $G$ and $G'$.
    Assume that there is a group epimorphism $G \to G'/\frat(G')$. Then, $G'\act \prim(G')$ elementary implies as loop condition $G \act \prim(G)$.
\end{lemma}
\begin{proof} We have
    \begin{align*}
        G'\act \prim(G')
        &\ecle G/\frat(G')\act \prim(G')&&\text{ (quotient action along }G'\to G'/\frat(G'))\\
        &\ecle G\act \prim(G')&&\text{ (subgroup action along }G\to G'/\frat(G'))\\
        &\ecle G\act \prim(G) &&\text{ ($G$-set homomorphism } \prim(G')\to \prim(G))
    \end{align*}
    showing the claim.
    Note that $G\to G'$ being surjective is used in the assumption that $G\act \prim(G')$ has no fixed point and thus there is a $G$-set homomorphism $\prim(G')\to \prim(G)$.
\end{proof}

\subsubsection{Combined}
We get the full classification for structures and conditions of the form $G\act \prim(G)$.
\begin{theorem} \label{TheoremFrattiniActionPoset}
    Let $G$ and $H$ be two groups. Then, the following are equivalent:
\begin{enumerate}
    \item \label{NrFrattiniActionPoset1}
    There is a group epimorphism $G\to H/\frat(H)$,
    \item \label{NrFrattiniActionPoset2}
    $\Pol_G(\prim(G)) \not\models \condition (H \act \prim(H))$,
    \item \label{NrFrattiniActionPoset3}
    $G\act \prim(G)$ ppg-constructs the action $H\act \prim(H)$.
    \item \label{NrFrattiniActionPoset4}
    $\condition(H\act \prim(H))$ elementary implies $\condition(G \act \prim(G))$,
\end{enumerate}
\end{theorem}
\begin{proof}
    \ref{NrFrattiniActionPoset2} implies \ref{NrFrattiniActionPoset1} by Lemma~\ref{LemmaPrimitiveClassificationToEpi}.

    \ref{NrFrattiniActionPoset3} implies \ref{NrFrattiniActionPoset2} by Corollary~\ref{CorPrimitiveClassificationOnlyIfDirection}.

    \ref{NrFrattiniActionPoset1} implies \ref{NrFrattiniActionPoset3} by Lemma~\ref{LemmaPrimitiveClassificationIfDirection}.

    \ref{NrFrattiniActionPoset4} implies \ref{NrFrattiniActionPoset2} by Lemma~\ref{LemmaConditionImpliesEpiShortened}.

    \ref{NrFrattiniActionPoset1} implies \ref{NrFrattiniActionPoset4} by Lemma~\ref{LemmaEpiImpliesCondition}.
\end{proof}

\subsection{Reducing general actions to primitive actions}
For two arbitrary group actions, $G\act X$ and $H\act Y$ the questions wether they ppg-construct each other and wether they elementary imply each other can be reduced to questions in terms of Frattini actions of subquotients, as we will show now.

\subsubsection{ppg-constructions}
We start again with ppg-constructions.
\begin{lemma}	\label{LemmaGeneralToPrimitiveA}
	The group action $G\act X$ ppg-constructs the action $H\act \prim(H)$ for all subgroups $H\le G$ such that $H\act X$ (with the subgroup action) is a minimal fixed point free action.
\end{lemma}
\begin{proof}
	Let $H\le G$ such that $H \act X$ is a minimal fixed-point-free action. Then by definition, $G\act X$ ppg-constructs $H \act X$ as this is a subaction. Since $H\act X$ is minimal fixed-point free, it is homomorphically equivalent to $H \act \prim(H)$. Thus, this action can also be constructed.
\end{proof}

\begin{remark}
    Note that if $G\act X$ is an action on a non-empty set then $\{1\}$ is never a minimal fixed-point free action as this group always acts with fixed point. If $G\act X$ has a fixed point then the set of all subgroups $S$ such that $S\act X$ is minimal fixed-point free, is empty.
\end{remark}

\begin{theorem} \label{TheoremConstructingGroupActionA}
    The following are equivalent for two group actions $G\act X$ and $G'\act X'$:
    \begin{enumerate}
        \item \label{NrTheoremConstructingGroupActionA1}
        $G\act X$ ppg-constructs $G'\act X'$
        \item \label{NrTheoremConstructingGroupActionA2}
        $G\act X$ ppg-constructs $H\act \prim(H)$ for all $H\le G'$ such that $H\act X'$ is minimal fixed-point free.
        \item  \label{NrTheoremConstructingGroupActionA3}
        $\Pol_G(X)\not\models \condition(H\act \prim(H))$ for all $H\le G'$ such that $H\act X'$ is minimal fixed-point free.
        \item  \label{NrTheoremConstructingGroupActionA4}
        $\Pol_G(X)\not\models \condition(H\act X')$ for all $H\le G'$ such that $H\act X'$ is minimal fixed-point free.
    \end{enumerate}
\end{theorem}
\begin{proof}
    \ref{NrTheoremConstructingGroupActionA1} implies \ref{NrTheoremConstructingGroupActionA2} as ppg-constructions are transitive by Lemma~\ref{LemmaPPConstructionsTransitive} and the fact that $G'\act X'$ ppg-constructs $H\act \prim(H)$ by Lemma~\ref{LemmaGeneralToPrimitiveA}.

    \ref{NrTheoremConstructingGroupActionA2} implies \ref{NrTheoremConstructingGroupActionA3} as ppg-constructions preserve loop conditions (Lemma~\ref{LemmaPPConstructionsPreserveMinorConditions}) and 
    $$\Pol_H(\prim(H))\not\models \condition(H\act \prim(H))$$
    by Lemma~\ref{LemmaNoSelfLoop} \cite[Lemma~4.16]{meyerStarke2024finitesimplegroups}.

    \ref{NrTheoremConstructingGroupActionA3} implies \ref{NrTheoremConstructingGroupActionA4}:
    Since $H\act X'$ is minimal fixed-point free, it is homomorphically equivalent to $H\act \prim(H)$, Lemma~\ref{LemmaMinimalFixedPointFreeHomEquiv}. Thus, the conditions $\condition(H\act X')$ and $\condition(H\act \prim(H))$ are equivalent, Lemma~\ref{LemmaElImplImplyImply}.
    
    \ref{NrTheoremConstructingGroupActionA4} implies \ref{NrTheoremConstructingGroupActionA1}:
    This is a generalization of \cite[Theorem 4.22]{meyerStarke2024finitesimplegroups}.

    Let $\Pol_G^{X'}(X)$ denote the $X'$-ary polymorphisms of $X$, note that $\Pol^{X'}(X)$ is a subset of $X^{X^{X'}}$. Consider the minor action on $\Pol_G^{X'}(X)$ which is a $G'$-action. 
    This action can be ppg-constructed from $G\act X$ by Lemma~\ref{LemmaPPConstructMinorAction}.
    We now show that $\Pol_G^{X'}(X)$ and $X'$ are homomorphically equivalent as $G'$-sets.

    It is easy to see that the map $X'\to \Pol_G^{X'}(X) \subseteq X^{X^{X'}}, x\mapsto \pi_x$, that maps an element $x\in X'$ to the projection on the $x$-component, i.e. the map $\pi_x\colon A^{X'} \to A, t\mapsto t_x$, is a homomorphism.

    To find a homomorphism $h\colon \Pol_G^{X'}(X) \to X'$, consider a non-isolated point $f\in \Pol_G^{X'}(X)$ and its connected component $C$ which admits a transitive $G'$-action.
    Let $\stab_{G'}(f)$ be the stabilizer of $f$ in $G'$. Note that $f$ is an $X'$-ary polymorphism that is invariant under $\stab_{G'}(f)$, thus $\Pol_G(X)\models \condition(H \act X')$ for all subgroups $H\le \stab_{G'}(f)$. By the assumption, $H\act X'$ cannot be a minimal fixed-point free action thus $\stab_{G'}(f) \act X'$ has a fixed-point. Define $h(f)$ as this fixed-point and $h(\minoraction{g}{f})$ as $g.h(f)$. This is easily seen to be a $G'$-set homomorphism $C \to X'$. As we can do this in every component, we get a homomorphism $\Pol_G^{X'}(X) \to X'$.

    Therefore, $X'$ is homomorphically equivalent to $\Pol_G^{X'}(X)$ and can be ppg-constructed from $G\act X$.
\end{proof}

\subsubsection{loop conditions}
The argument for loop conditions is easier this time.
\begin{lemma} \label{LemmaOneStep}
    Let $H\act Y$ be a group action without fixed point. Then, $(H\act Y)$ is as condition elementary equivalent to
    $\{H \act \prim(H)\}\cup\{(M\act Y\mid M\le H\text{ maximal subgroup}\}$.
\end{lemma}
\begin{proof}
    We apply Lemma~\ref{LemmaConditionsCompose} with $Y=\prim(H)$.
\end{proof}

\begin{corollary} \label{CorollaryConditionToPrimitive}
    Let $H\act Y$ be a group action. Then, $H\act Y$ is as condition elementary equivalent to
    $$\{(S\act \prim(S))\mid S\le H\text{ such that }S\act Y \text{ fixed-point free}\}.$$
\end{corollary}
\begin{proof}
    If $H\act Y$ has a fixed point, then both conditions are trivial.
    By Lemma~\ref{LemmaOneStep}, we get that 
    $$
        \condition(H\act Y) \eceq \{H \act \prim(H)\}\cup\{(M\act Y\mid M\le H\text{ maximal subgroup}\}.
    $$
    We can remove all groups $M$ that have a fixed point and inductively apply Lemma~\ref{LemmaOneStep} again on the others until all groups are trivial. Note that if $S\le H$ has no fixed point in $Y$, also every supergroup of $S$ has no fixed point.
    Thus, we get exactly a condition $S \act \prim(S)$ for all subgroups $S$ without a fixed point. 
\end{proof}

\subsection{Classification of general actions}
\label{SectionClassificationsGeneral}
We have now proven all parts to describe the full classification.

    \begin{figure}
        \centering
 	\begin{tikzpicture}
		[group/.style={rounded corners=3mm,draw,thick},
		cloud/.style={circle,y radius = 5cm,draw,thick},
		]
		\node (A) at ( -1,0) [group] {$\IZ/1$};
		\node (B1) at ( -3,-1) [group] {$\IZ/2$};
		\node (B2) at ( -1,-1) [group] {$\IZ/3$};
		\node (B3) at (  1,-1) [group] {$\IZ/5$};
		\node (B4) at (  2,-1) [group] {$\IZ/7$};
		\node (B5) at (  3,-1) [group] {$A_5$};
		\node (C0) at ( -4.5,-2) [group] {$(\IZ/2)^2$};
		\node (C1) at ( -3.4,-2) [group] {$S_3$};
		\node (C2) at ( -2.5,-2) [group] {$S_5$};
		\node (C3) at ( -1.5,-2) [group] {$\IZ/6$};
		\node (C4) at ( -0.6,-2) [group] {$A_4$};
		\node (C5) at ( +0.6,-2) [group] {$(\IZ/3)^2$};
		\draw (A.south)--(B1.north);
		\draw (A.south)--(B2.north);
		\draw (A.south)--(B3.north);
		\draw (A.south)--(B4.north);
		\draw (A.south)--(B5.north);
		\draw (B1.south)--(C0.north);
		\draw (B1.south)--(C1.north);
		\draw (B1.south)--(C2.north);
		\draw (B1.south)--(C3.north);
		\draw (B2.south)--(C3.north);
		\draw (B2.south)--(C4.north);
		\draw (B2.south)--(C5.north);
		\draw[dashed] (-5,0) .. controls (-4,-2) and (-1,-2) .. 
		(2,0);
        \draw node at (-3.5,-0.2) { $\ppchar(\IZ/6\act \IZ/6)$};
		\draw[dotted] (-5.2,-1.8) .. controls (-4,0) and (-1,0) .. 
		(2,-2);
        \draw node at (3,-2.3) { $\condchar(\IZ/6\act \IZ/6)$};
	\end{tikzpicture}
       \caption{The epimorphism poset on some finite groups with trivial Frattini subgroup. The group action $\IZ/6\act \IZ/6$ is ppg-interconstructable with the product of $G\act \prim(G)$ where $G$ is inside $\ppchar(\IZ/6\act \IZ/6)$ that is the set of all groups above the dashed line.
       The condition $\condition(\IZ/6\act \IZ/6)$ is satisfied whenever $\condition(G\act \prim(G))$ is satisfied for all groups in $\condchar(\IZ/6\act \IZ/6)$ that are all groups below the dotted line. As this set is infinite, we consider the set $\condcharred(\IZ/6\act \IZ/6)$ instead which consists only of the finitely many elements directly below the dotted line.
       }
        \label{figure_epiPoset}
    \end{figure}
	
Both the ppg-construction poset and the elementary implication poset can be seen as comparison of subsets of the epimorphism poset of groups, see Figure~\ref{figure_epiPoset} for an example. The details will be explained in this section.

\subsubsection{ppg-constructions}
We start as usually with the ppg-constructions. 
\begin{definition}
	For a group action $G\act X$, let its ppg-construction characteristic $\ppchar(G\act X)$ be the set of all groups $G'$ such that there exists $N\normalsubgroup H\le G$ and
	\begin{enumerate}
		\item $G'$ is isomorphic to the subquotient $H/N$.
		\item The subgroup action $H\act X$ is a minimal fixed point free action.
		\item $\frat(G')$ is trivial.
	\end{enumerate}
\end{definition}

\begin{remark}
    As $G'$ needs to be a subquotient in the previous definition, $\ppchar(G\act X)$ contains only a finite number of groups up to isomorphism. As quotients of groups with trivial Frattini group also have trivial Frattini group, this set is also closed under epimorphisms.
    
    It can be any finite set of groups with trivial Frattini subgroup that is closed under epimorphisms.

    Precisely speaking, $\ppchar(G\act X)$ is a proper class. However, we usually consider it as finite set by identifying isomorphic groups.

    It is easy to see that 
    $$\ppchar(G\act X) = \ppchar(\prod_{G'\in \ppchar(G\act X)}(G'\act \prim(G')))  $$   
    as maximal fixed-point free subgroups of the product are maximal fixed-point free subgroups of a factor.
\end{remark}
\begin{lemma} \label{LemmaPPCharCharacterisation}
    Let $G\act X$ be a group action. Then, the following classes of groups are equal:
    \begin{enumerate}
        \item $\ppchar(G\act X)$
        \item $\{H\mid \frat(H)=\{1\} \land G\act X \ppgle H\act\prim(H)\}$
        \item $\{H\mid \frat(H)=\{1\} \land \Pol_G(X) \not\models \condition(H\act\prim(H))\}$
    \end{enumerate}
\end{lemma}
\begin{proof}
    Each of the groups in $\ppchar(G\act X)$ have trivial Frattini subgroup by definition. Thus, this condition is in all three descriptions and can be ignored.

    The first class is contained in the second: By Theorem~\ref{TheoremConstructingGroupActionA}, implication \ref{NrTheoremConstructingGroupActionA1} to \ref{NrTheoremConstructingGroupActionA2}, $G_1\act X_1$ ppg-constructs $H\act \prim(H)$ for all subgroups $H\le G_2$ that are minimal fixed-point free. By Lemma~\ref{LemmaPrimitiveClassificationIfDirection}, they also ppg-construct all group actions $H\act \prim(H)$ for $H\in \ppchar(G_2\act X_2)$.

    The second class is contained in the third by Lemma~\ref{LemmaNoSelfLoop} as pp-constructions preserve loop conditions.
    
    The third class is contained in the first one: 
    Consider any $H$ not in $\ppchar(G\act X)$ and any $S\in \ppchar(G\act X)$. Then, we get 
    $$
        \Pol_S\prim(S)\models\condition(H \act \prim(H))
    $$
    by Lemma~\ref{LemmaPrimitiveClassificationToEpi}.
    Consider now the external product action $\prod_{S \in \ppchar(G\act X)} S \act \prod_{S \in \ppchar(G\act X)} \prim(S)$. This action admits a $\condition(H \act \prim(H))$-polymorphism, by applying the maps component wise. Moreover, it ppg-constructs $G\act X$ by Theorem~\ref{TheoremConstructingGroupActionA}. Thus, we get 
    $\Pol_G(X) \models \condition(H\act\prim(H))$ and $H$ is not in the third set.
\end{proof}

\begin{lemma}
    A group action $G\act X$ is ppg-interconstructable with the product $\prod_{G'\in \ppchar(G)} (G'\act \prim(G'))$.
\end{lemma}
\begin{proof}
    This is a combination of Theorem~\ref{TheoremConstructingGroupActionA} and Theorem~\ref{TheoremFrattiniActionPoset}.
\end{proof}

\begin{theorem}[Classification of ppg-constructability]
    \label{TheoremMainPPConstruction}
    For two group actions $G_1\act X_1$ and $G_2\act X_2$, the following are equivalent:
\begin{enumerate}
        \item \label{NrMainPPConstruction1}
        $G_1\act X_1$ ppg-constructs $G_2\act X_2$.
        \item \label{NrMainPPConstruction4}
        $\ppchar(G_2\act X_2)\subseteq \ppchar(G_1\act X_1)$.
        \item \label{NrMainPPConstruction2}
        $\Pol_{G_2}(X_2)$ satisfies every loop condition satisfied by $\Pol_{G_1}(X_1)$.
        \item \label{NrMainPPConstruction3}
        $G_1\act X_1$ ppg-constructs $H\act \prim(H)$ for all $H\in \ppchar(G_2\act X_2)$.
        \item \label{NrMainPPConstruction5}
        every subquotient of $G_2$, where the subgroup acts minimal fixed-point free on $X_2$, that has a trivial Frattini subgroup is isomorphic to a subquotient of $G_1$, where the subgroup acts minimal fixed-point free on $X_1$.
        \item \label{NrMainPPConstruction6}
        $\Pol_{G_2}(X_2)$ satisfies every loop condition satisfied by $\Pol_{G_1}(X_1)$ that is of the form $\condition(H\act \prim(H))$ for a group $H$ with additionally $\frat(H)=\{1\}$.
    \end{enumerate}
\end{theorem}
\begin{proof}
    \ref{NrMainPPConstruction1} implies \ref{NrMainPPConstruction2} by Lemma~\ref{LemmaPPConstructionsPreserveMinorConditions}.

    \ref{NrMainPPConstruction1} implies \ref{NrMainPPConstruction3} by Lemma~\ref{LemmaPPCharCharacterisation} as ppg-constructable is a transitive relation.

    \ref{NrMainPPConstruction3} implies \ref{NrMainPPConstruction4} by Lemma~\ref{LemmaPPCharCharacterisation}.

    \ref{NrMainPPConstruction4} is equivalent to \ref{NrMainPPConstruction5} by definition.

    \ref{NrMainPPConstruction2} implies \ref{NrMainPPConstruction6} is clear.

    \ref{NrMainPPConstruction4} implies \ref{NrMainPPConstruction6} by Lemma~\ref{LemmaPPCharCharacterisation} and a contraposition.

    \ref{NrMainPPConstruction6} implies \ref{NrMainPPConstruction1}: We know that $G_2\act X_2$ ppg-constructs itself. By Theorem~\ref{TheoremConstructingGroupActionA}, implication \ref{NrTheoremConstructingGroupActionA1} to \ref{NrTheoremConstructingGroupActionA3}, we get that 
    $$\Pol_{G_2}(X_2) \not\models \condition(H \act \prim(H))$$
    for all $H\le G_2$ such that $H\act X_2$ is minimal fixed-point free. 
    By assumption \ref{NrMainPPConstruction6}, we get that these conditions are also not true in $\Pol_{G_1}(X_1)$. 
    Thus, we can apply Theorem~\ref{TheoremConstructingGroupActionA}, implication \ref{NrTheoremConstructingGroupActionA3} to \ref{NrTheoremConstructingGroupActionA1}, and get that $G_1\act X_1$ ppg-constructs $G_2\act X_2$.
\end{proof}
\begin{theorem}[Classification of group actions up to ppg-interconstructability]
    The ppg-interconstructability defines an equivalence relation on finite group actions.
    Two group actions $G_1\act X_1$ and $G_2\act X_2$ ppg-construct each other if and only if $\ppchar(G_2\act X_2) = \ppchar(G_1\act X_1)$.

    A representation system of these equivalence classes is given by group actions of the form $\prod_{G'\in M} G' \act \prod_{G'\in M}\prim(G')$ where $M$ is a finite set of finite groups with trivial Frattini subgroup that is closed under epimorphisms. The representative of the equivalence class of $G\act X$ is given by the external direct product $\prod_{G'\in \ppchar(G\act X)} G' \act \prod_{G'\in \ppchar(G\act X)}\prim(G'))$.
\end{theorem}
\begin{proof}
    The ppg-interconstructability is clearly reflexive and symmetric. It is transitive by Lemma~\ref{LemmaPPConstructionsTransitive} and thus an equivalence relation. It is equivalent to the equality on $\ppchar$ as ppg-constructability translates to inclusion by Theorem~\ref{TheoremMainPPConstruction}.

    The representation system follows as $G\act X$ ppg-constructs $G'\act \prim(G')$ for all $G\in \ppchar(G\act X)$ by Lemma~\ref{LemmaPPCharCharacterisation} and thus also their product. Conversely, this product ppg-constructs $G\act X$ by Theorem~\ref{TheoremMainPPConstruction}, implication \ref{NrMainPPConstruction3} to \ref{NrMainPPConstruction1}.
\end{proof}

\begin{corollary}
    Every group action ppg-constructs only finitely many group actions up to ppg-interconstructability.
\end{corollary}
\begin{proof}
    If $G\act X$ ppg-constructs $H\act Y$, then we get that $\ppchar(H\act Y)$ is one of the finitely many subsets of $\ppchar(G\act X)$. Thus, there is only a finite number of such actions $H\act Y$, that are pairwise non-ppg-interconstructable.
\end{proof}

\subsubsection{loop conditions}
A similar claim holds for the loop conditions.

\begin{definition}
    For a group action $H\act Y$ let its reduced loop condition characteristic $\condcharred(H\act Y)$ be the class of all groups isomorphic to $S/\frat(S)$ such that
    \begin{enumerate}
        \item $S\le H$ is a subgroup such that $S\act Y$ has no fixed point and
        \item there is no subgroup $S'\le H$ such that $S'\act Y$ has no fixed point and there exists an epimorphism $S/\frat(S)\to S'/\frat(S')$ which is not an isomorphism.
    \end{enumerate}
\end{definition}
\begin{remark}
    The class $\condcharred(H\act Y)$ contains only a finite number of groups up to isomorphism. Thus, we usually consider it as finite set even though it is a proper class.
    Every group in this set has a trivial Frattini subgroup and no group in this set admits an epimorphism into another group in this set.

    In contrast to $\ppchar$, which is closed under ppg-constructions, we decided to choose $\condcharred$ not to be closed under elementary implications as this would result in an infinite set which we denote $\condchar$. Instead, $\condcharred$ is reduced, that is it is designed as the minimal representation of its equivalence class, that is the inclusion smallest set, which happens to exists.

    Depending on $H$ and $Y$, the set $\condcharred(H\act Y)$ can be any finite set of groups with trivial Frattini subgroup such that no group in this set admits an epimorphism into another group in this set. 
\end{remark}
\begin{lemma}   \label{LemmaCondCharCharacterisation}
    Let $H\act Y$ be a group action. Then, the following classes of groups are equal:
    \begin{enumerate}
        \item $\{G\mid \exists N\normalsubgroup G: G/N\in \condcharred(H\act Y)\}$,
        \item $\{G\mid \Pol_G(\prim(G)) \not\models \condition(H\act Y)\}$, and
        \item $\{G \mid \condition(H\act Y) \ecle \condition(G\act \prim(G))\}$.
    \end{enumerate}
\end{lemma}
\begin{proof}
    The first class is contained in the third as for all $G$ in the class set, there is a fixed-point free subgroup $S\le H$ and $N\normalsubgroup G$ such that $S/\frat(S)\iso G/N$ and we have
    \begin{align*}
        \condition (H\act Y )
        &
        \ecle \condition (S\act Y )
        \eceq \condition (S \act \prim(S) )
        \\&
        \eceq \condition (S/\frat(S) \act \prim(S/\frat(S)))
        \eceq \condition (G/N \act \prim(G/ N) )
        \\&
        \eceq \condition (G \act \prim(G/N) )
        \ecle \condition (G \act \prim(G)).
    \end{align*}

    The third class is contained in the second:
    If $G$ would be in the third, but not in the second class, then $\Pol_G(\prim(G))$ would have an element invariant under $H\act Y$ and thus by the third condition an element invariant under $G\act \prim(G)$. This contradicts Lemma~\ref{LemmaNoSelfLoop}.

    The second class is contained in the first: By Corollary~\ref{CorollaryConditionToPrimitive}, $H\act Y$ is elementary equivalent and thus equivalent to $S\act \prim(S)$ for all subgroups $S\le H$ that act fixed-point free. Thus, for each $G$ in the second class, there is such a subgroup $S$ such that
    $$
        \Pol_G(\prim(G)) \not\models S\act \prim(S) \eceq S/\frat(S)\act \prim(S/\frat(S))
    $$
    We may assume that $|S/\frat(S)|$ is minimal and get by Lemma~\ref{LemmaEpiImpliesCondition} that $S/\frat(S)\in \condcharred(H\act Y)$. 
    Similarly, we get by Lemma~\ref{LemmaPrimitiveClassificationToEpi}, that there is an epimorphism $G\to S/\frat(S)$. Thus, $\exists N\normalsubgroup G: G/N\in \condcharred(H\act Y)$.
\end{proof}

\begin{lemma} \label{LemmaCondcharredUnfold}
    Let $H\act Y$ be a group action. Then, $\condition(H\act Y)$ is elementary equivalent to $\{\condition(H'\act \prim(H')) \mid H'\in \condcharred(H\act Y)\}$.
\end{lemma}
\begin{proof}
    This is a combination of Corollary~\ref{CorollaryConditionToPrimitive} and Theorem~\ref{TheoremFrattiniActionPoset}.
\end{proof}

\begin{theorem}[Classification of elementary implications.]
    \label{TheoremClassificationEI}
    Let $H_1\act Y_1$ and $H_2\act Y_2$ be group actions. The following are equivalent:
    \begin{enumerate}
        \item \label{NrTheoremClassificationEI1}
        $H_1\act Y_1\ecle H_2\act Y_2$.
        \item \label{NrTheoremClassificationEI3}
        Every group in $\condcharred(H_2\act Y_2)$ admits an epimorphism into a group inside $\condcharred(H_1\act Y_1)$.
        \item \label{NrTheoremClassificationEI2}
        For every group action $G\act X$ such that $\Pol_G(X)\models \condition(H_1\act Y_1)$, we also have $\Pol_G(X)\models \condition(H_2\act Y_2)$.
        \item \label{NrTheoremClassificationEI5}
        For all groups $S$ with trivial Frattini subgroup, we have
        $$
            \Pol_S(\prim(S)) \models H_1\act Y_1 \implies \Pol_S(\prim(S)) \models H_2\act Y_2.
        $$
        \item \label{NrTheoremClassificationEI6}
        For all groups $S\in \condcharred(H_2\act Y_2)$, we have 
        $$
            \Pol_S(\prim(S)) \not\models H_1\act Y_1.
        $$
        \item \label{NrTheoremClassificationEI4}
        For every subgroup $S_2\le H_2$, such that $S_2\act Y_2$ has no fixed-point, there exists $S_1\le H_1$ and a normal subgroup $N_2\normalsubgroup S_2$ such that $S_1\act Y_1$ has no fixed-point and $S_1/\frat(S_1) \iso S_2/N_2$.
    \end{enumerate}
\end{theorem}
\begin{proof}
    \ref{NrTheoremClassificationEI1} implies \ref{NrTheoremClassificationEI2} by Lemma~\ref{LemmaElImplImplyImply}.

    \ref{NrTheoremClassificationEI2} implies \ref{NrTheoremClassificationEI5} is obvious.

    \ref{NrTheoremClassificationEI5} implies \ref{NrTheoremClassificationEI6} as all groups in $\condcharred(H_2\act Y_2)$ have a trivial Frattini subgroup.

    \ref{NrTheoremClassificationEI6} is equivalent to  \ref{NrTheoremClassificationEI3} by 
    Lemma~\ref{LemmaCondCharCharacterisation}. We get that $\Pol_S(\prim(S)) \not\models H_1\act Y_1$, if and only if $S$ has am epimorphism into a group in $\condcharred(H_1\act X_1)$. 

    \ref{NrTheoremClassificationEI3} implies \ref{NrTheoremClassificationEI4}: Every such subgroup $S_2$ has an epimorphism into $\condcharred(H_2\act X_2)$ by Lemma~\ref{LemmaCondCharCharacterisation}. By assumption, it moreover has an epimorphism into a group in $\condcharred(H_1\act X_1)$ which we can write as $S_1/\frat(S_1)$ for $S_1\act Y_1$ without fixed point. So there is $N_2\normalsubgroup S_2$ such that $S_2/N_2\iso S_1/\frat(S_1)$.

    \ref{NrTheoremClassificationEI4} implies \ref{NrTheoremClassificationEI1} as
    \begin{align*}
        H_1\act Y_1
        &\ecle \{ S \act \prim(S) \mid S\le H_1, S\act Y_1 \text{ fixed point free}\}
        \\&\ecle \{ S \act \prim(S) \mid S\le H_2, S\act Y_2 \text{ fixed point free}\}
        \\&\ecle H_2\act Y_2
    \end{align*}
    by Corollary~\ref{CorollaryConditionToPrimitive} and Lemma~\ref{LemmaEpiImpliesCondition} and again Corollary~\ref{CorollaryConditionToPrimitive}.
\end{proof}

\begin{theorem}[Classification of loop conditions up to elementary equivalence.]
    The existence of an elementary implication in both direction defines an equivalence relation on loop conditions.
    Two gloop conditions $\condition(G_1\act X_1)$ and $\condition(G_2\act X_2)$ are elementary equivalent, if $\condcharred(G_2\act X_2) = \condcharred(G_1\act X_1)$.

    A representation system of these equivalence classes is given by loop conditions of the form $\condcharred(\prod_{G'\in M} G' \act \prod_{G'\in M}\prim(G'))$ where $M$ is a finite set of finite groups with trivial Frattini subgroup such that no group in $M$ has an epimorphism into another group in $M$. 
    
    The representative of the equivalence class of $G\act X$ is given by the external direct product $\prod_{G'\in \condcharred(G\act X)} G' \act \prod_{G'\in \condcharred(G\act X)}\prim(G'))$.

    Moreover, $\condition(G\act X)$ is elementary equivalent to the set 
    $$\{\condition(H\act \prim(H))\mid H \in \condcharred(G\act X)\} .$$
\end{theorem}

\subsubsection{Combined}
We also have now some easy to check equivalent conditions for $\Pol_G(X)\models \condition(H\act Y)$.

\begin{theorem}[Classification of modeling a loop condition.] \label{TheoremModelsCharakterisierung}
    Let $G\act X$ and $H\act Y$ be group actions. The following are equivalent:
    \begin{enumerate}
        \item \label{NrTheoremModelsCharakterisierung1}
        $\Pol_G(X)\not \models \condition( H\act Y)$.
        \item \label{NrTheoremModelsCharakterisierung4}
       	$\condcharred(H\act Y)$ and $\ppchar(G\act X)$ are not disjoint.
        \item \label{NrTheoremModelsCharakterisierung2a}
        $G\act X$ ppg-constructs $H'\act \prim(H')$ for some group $H'\in \condcharred(H\act Y)$.
        \item \label{NrTheoremModelsCharakterisierung2b}
        The condition $\condition (H\act Y)$ elementary implies $\condition (G'\act \prim(G'))$ for some $G'\in \ppchar(G\act X)$.
        \item \label{NrTheoremModelsCharakterisierung5}
        There are subgroups $G''\normalsubgroup G' \le G$ and $H''\normalsubgroup H' \le H$ such that $G'\act X$ is minimal fixed-point free and $H'\act Y$ is fixed-point free and $G'/G'' \iso H'/H''$.
    \end{enumerate}
\end{theorem}
\begin{proof}
    \ref{NrTheoremModelsCharakterisierung1} implies \ref{NrTheoremModelsCharakterisierung2a}: $\condition(H\act Y)$ is elementary equivalent to $\{ \condition(H'\act \prim(H')) \mid 
    H'\in \condcharred(H\act Y)$ by Lemma \ref{LemmaCondcharredUnfold}. By Theorem~\ref{TheoremConstructingGroupActionA}, applied to $G'\act X'=H'\act \prim(H')$, this is equivalent to $G\act X$ ppg-constructing $H'\act \prim(H')$ for some group $H'\in \condcharred(H\act Y)$.

    \ref{NrTheoremModelsCharakterisierung2a} is equivalent to \ref{NrTheoremModelsCharakterisierung4} by Lemma~\ref{LemmaPPCharCharacterisation}.

    \ref{NrTheoremModelsCharakterisierung4} implies \ref{NrTheoremModelsCharakterisierung2b} by Lemma~\ref{LemmaCondcharredUnfold}.
    
    \ref{NrTheoremModelsCharakterisierung2b} implies \ref{NrTheoremModelsCharakterisierung4}: By Lemma~\ref{LemmaCondCharCharacterisation}, there is an element in $\ppchar(G\act X)$ that is a quotient of an element in $\condcharred(H\act Y)$. As $\condcharred(H\act Y)$ is closed under quotients by definition, it is also inside this set.

    \ref{NrTheoremModelsCharakterisierung4} implies \ref{NrTheoremModelsCharakterisierung5} as the common element of $\condcharred(H\act Y)$ and $\ppchar(G\act X)$ is a subquotient of both $G$ and $H$ satisfying the conditions.

    \ref{NrTheoremModelsCharakterisierung5} implies \ref{NrTheoremModelsCharakterisierung1} as $\Pol_G(X)\not \models \condition( G'/G'')$ and this is elementary implied by $\condition( H\act Y)$.
\end{proof}

\section{Group Actions and Loop Conditions in the context of Universal algebra}
\label{SectionApplyUniversalAlgebra}
We have considered in Section~\ref{SectionClassification} ppg-constructions and elementary implications in the context of group actions and loop conditions. We now translate these constructions into the general setting of universal algebra as introduced in Section~\ref{SectionPreliminariesUniversalAlgebra} 
consisting of first order structures and minor conditions.

\subsection{ppg-constructions}
We have seen that ppg-constructions can be expressed in terms of loop conditions. We show that this notion is equivalent to pp-constructions and also agrees the description for minor conditions.

\begin{theorem} \label{TheoremPPConstructionGroups}
    Let $G_1\act X_1$ and $G_2\act X_2$ be two (finite) group actions. Then, the following are equivalent:
    \begin{enumerate}
        \item \label{NrppConstruct1}
        The group action $G_1\act X_1$ ppg-constructs the group action $G_2\act X_2$.
        \item \label{NrppConstruct2}
        The structure $\structure(G_1\act X_1)$ pp-constructs the structure $\structure(G_2\act X_2)$.
        \item \label{NrppConstruct3}
        The structure $\structure(G_2\act X_2)$ satisfies all minor conditions of $\structure(G_1\act X_1)$.
        \item \label{NrppConstruct4}
        The structure $\structure(G_2\act X_2)$ satisfies all loop conditions of the form $H\act \prim(H)$ of $\structure(G_1\act X_1)$.
        \item \label{NrppConstruct5}
        For all groups $H$ we have that $\Pol_{G_1}(X_1)\models \condition(H\act \prim(H))$ implies $\Pol_{G_2}(X_2)\models \condition(H\act \prim(H))$.
        \item \label{NrppConstruct6}
        $\ppchar(G_2 \act X_2) \subseteq \ppchar(G_1\act X_1)$
    \end{enumerate}
\end{theorem}
\begin{proof}
    \ref{NrppConstruct1} implies \ref{NrppConstruct2} as ppg-constructions are a special kind of pp-constructions by definition.

    \ref{NrppConstruct2} implies \ref{NrppConstruct3} by Theorem~\ref{TheoremReflections}. 

    \ref{NrppConstruct3} implies %\ref{NrppConstruct4} and \ref{NrppConstruct4} implies 
    \ref{NrppConstruct4} is obvious.

    \ref{NrppConstruct4} is equivalent to \ref{NrppConstruct5} as $\Pol(\structure(G\act X))=\Pol_G(X)$.

    \ref{NrppConstruct5} implies \ref{NrppConstruct6} and \ref{NrppConstruct6} implies \ref{NrppConstruct1} by Theorem~\ref{TheoremMainPPConstruction}. 
\end{proof}

\subsection{loop conditions}
We have seen in Section~\ref{SectionClassification} that elementary implications of loop conditions are equivalent to implications as condition on finite groups. Here, we show that it is also equivalent to being an implication as condition on general first order structures.

\begin{theorem} \label{TheoremElementImplication}
    Let $H_1\act Y_1$ and $H_2\act Y_2$ be two (finite) group actions. Then, the following are equivalent:
    \begin{enumerate}
        \item \label{NrElementImplication1}
        The condition $\condition(H_1\act Y_1)$ elementary implies the condition $\condition(H_2\act Y_2)$.
        \item \label{NrElementImplication2}
        The condition $\condition(H_1\act Y_1)$ implies the condition $\condition(H_2\act Y_2)$ for all clones.
        \item \label{NrElementImplication3}
        The condition $\condition(H_1\act Y_1)$ implies the condition $\condition(H_2\act Y_2)$ for the polymorphism minion of all (possibly infinite) structures.
        \item \label{NrElementImplication4}
        The condition $\condition(H_1\act Y_1)$ implies the condition $\condition(H_2\act Y_2)$ for the polymorphism minion of all finite structures.
        \item \label{NrElementImplication5}
        The condition $\condition(H_1\act Y_1)$ implies the condition $\condition(H_2\act Y_2)$ for the polymorphism minion of $\structure(G\act X)$ where $G\act X$ is a finite permutation group.
        \item \label{NrElementImplication6}
        The condition $\condition(H_1\act Y_1)$ implies the condition $\condition(H_2\act Y_2)$ for the polymorphism minion of $\structure(G\act \prim(G))$ where $G$ is a finite group with trivial Frattini subgroup.
        \item \label{NrElementImplication7}
        Every group in $\condchar(H_2\act Y_2)$ admits an epimorphism into a group inside $\condchar(H_1\act Y_1)$.
    \end{enumerate}
\end{theorem}

\begin{proof}
    \ref{NrElementImplication1} implies \ref{NrElementImplication2} similar to Lemma~\ref{LemmaElImplImplyImply}.

    \ref{NrElementImplication2} implies \ref{NrElementImplication3} and \ref{NrElementImplication3} implies \ref{NrElementImplication4} are obvious.

    \ref{NrElementImplication4} implies \ref{NrElementImplication5} as a group action can be seen as a finite structure.

    \ref{NrElementImplication5} implies \ref{NrElementImplication6} is obvious.

    \ref{NrElementImplication6} implies \ref{NrElementImplication7} and \ref{NrElementImplication7} implies \ref{NrElementImplication1} is shown in Theorem~\ref{TheoremClassificationEI}.
\end{proof}

\subsection{General Structures}
In addition to describing dependencies of group actions, we can also link group actions with general structures. 
\begin{lemma} \label{LemmaBlockerStructure}
    Let $\structA$ be a finite structure and $G$ a finite group. Then, the following are equivalent:
    \begin{enumerate}
        \item The structure $\structA$ pp-constructs $\structure(G\act \prim(G))$.
        \item The minion $\Pol(\structA)$ does not satisfy $\condition(G\act \prim(G))$.
    \end{enumerate}
\end{lemma}
\begin{proof}
    The proof is similar to the proof of Theorem~\ref{TheoremConstructingGroupActionA} and we decided to omit it.
\end{proof}

This duality result extends to arbitrary group actions as written in the next two theorems.

\begin{theorem} \label{TheoremStructureConstructsGroup}
    Let $\structA$ be a finite structure and $G\act X$ some group action. Then, the following are equivalent:
    \begin{enumerate}
        \item \label{nrTheoremStructureConstructsGroup1}
        The structure $\structA$ pp-constructs $\structure(G \act X)$.
        \item \label{nrTheoremStructureConstructsGroup2}
        The structure $\structA$ pp-constructs $\prod_{G'\in \ppchar(G\act X)}(\structure(G'\act \prim(G')))$.
        \item \label{nrTheoremStructureConstructsGroup3}
        The structure $\structA$ pp-constructs $\structure(G'\act \prim(G'))$ for \emph{all} $G'\in \ppchar(G\act X)$.
        \item \label{nrTheoremStructureConstructsGroup4}
        The minion $\Pol(\structA)$ satisfies $\structure(G'\act \prim(G'))$ for \emph{no} $G'\in \ppchar(G\act X)$.
    \end{enumerate}
\end{theorem}
\begin{proof}
    \ref{nrTheoremStructureConstructsGroup1}$ \iff $\ref{nrTheoremStructureConstructsGroup2}:
    as the structures $\structA$ and $\prod_{G'\in \ppchar(G\act X)}(\structure(G'\act \prim(G')))$ are pp-interconstructable by Theorem~\ref{TheoremPPConstructionGroups}.

    \ref{nrTheoremStructureConstructsGroup2}$ \iff $\ref{nrTheoremStructureConstructsGroup3} as a structure $\structA$ pp-constructs a product if and only if it pp-constructs every factor.

    \ref{nrTheoremStructureConstructsGroup3}$ \iff $\ref{nrTheoremStructureConstructsGroup4} by Lemma~\ref{LemmaBlockerStructure}.
\end{proof}

\begin{theorem} \label{TheoremStructureModelsGroup}
    Let $\structA$ be a finite structure and $H\act Y$ some group action. Then, the following are equivalent:
    \begin{enumerate}
        \item \label{nrTheoremStructureModelsGroup1}
        The minion $\Pol(\structA)$ satisfies $\condition(H\act Y)$.
        \item \label{nrTheoremStructureModelsGroup2}
        The structure $\structA$ pp-constructs $\structure(H'\act \prim(H'))$ for \emph{no} $H'\in \condcharred(H\act Y)$.
        \item \label{nrTheoremStructureModelsGroup3}
        The minion $\Pol(\structA)$ satisfies $\structure(H'\act \prim(H'))$ for \emph{all} $H'\in \condcharred(H\act Y)$.
    \end{enumerate}
\end{theorem}
\begin{proof}
    \ref{nrTheoremStructureModelsGroup1}$\iff$\ref{nrTheoremStructureModelsGroup2} as these two sets of conditions are elementary equivalent by \ref{TheoremPPConstructionGroups} and this suffices by Theorem~\ref{TheoremElementImplication}.

    \ref{nrTheoremStructureModelsGroup2}$\iff$\ref{nrTheoremStructureModelsGroup3} by Lemma~\ref{LemmaBlockerStructure}.
\end{proof}

\subsection{Applications}
This paper brings a new insight, why finite simple groups showed up in the primitive positive constructability poset: In fact, the poset of all groups with trivial Frattini subgroup ordered by group epimorphisms is a subposet of the poset of all finite structures ordered by pp-constructions. As the simple groups are on top of this list, they appeared in \cite{meyerStarke2024finitesimplegroups}. 

In addition, we are now able to give a different proof of a variation of \cite[Theorem 4.10]{meyerStarke2024finitesimplegroups} by using the classification of all group actions up to elementary implications to improve \cite[Theorem 2]{BLP}:
\begin{theorem} \label{TheoremBLPBySimpleGroups}
    For a structure $\structA$, the following is equivalent:
    \begin{enumerate}
        \item \label{nrTheoremBLPBySimpleGroups2}
        The structure $\structA$ satisfies $\condition(S_n\act [n])$ for all $n$.
        \item \label{nrTheoremBLPBySimpleGroups3}
        The structure $\structA$ satisfies $\condition(G\act X)$ for all fixed-point free group actions $G\act X$.
        \item \label{nrTheoremBLPBySimpleGroups4}
        The structure $\structA$ satisfies $\condition(G\act \prim(G))$ for all finite simple groups $G$.
        \item \label{nrTheoremBLPBySimpleGroups5}
        The structure $\structA$ cannot pp-construct $\structure(G\act X)$ for any fixed-point free group actions $G\act X$.
    \end{enumerate}
    Moreover, it it known from the literature, that the previous conditions are equivalent to
    \begin{enumerate}[resume]
        \item \label{nrTheoremBLPBySimpleGroups1a}
        The constraint satisfaction problem CSP($\structA$) can be solved by basic linear programming robustly in polynomial time.
        \item The structure $\structA$ has width 1.
        \item The structure $\structA$ has a set operation.
        \item The structure $\structA$ has a measure operation.
        \item \label{nrTheoremBLPBySimpleGroups1e}
        The structure $\structA$ has tree duality.
    \end{enumerate}
\end{theorem}
\begin{proof}
    \ref{nrTheoremBLPBySimpleGroups2} implies \ref{nrTheoremBLPBySimpleGroups4} as $G\act \prim(G)$ is a subgroup action of $S_{|\prim(G)|}\act [|\prim(G))|]$ and thus the condition $\condition(G\act \prim(G))$ is implied by the condition $\condition( S_{|\prim(G)|}\act [|\prim(G))|] )$.

    \ref{nrTheoremBLPBySimpleGroups4} implies \ref{nrTheoremBLPBySimpleGroups3} by Theorem~\ref{TheoremStructureConstructsGroup}.

    \ref{nrTheoremBLPBySimpleGroups3} implies \ref{nrTheoremBLPBySimpleGroups2} as $S_n\act[n]$ has no fixed-point.

    \ref{nrTheoremBLPBySimpleGroups4} is equivalent to \ref{nrTheoremBLPBySimpleGroups5} by Lemma~\ref{LemmaBlockerStructure}.

    Finally, the equivalence of \ref{nrTheoremBLPBySimpleGroups2} and all of \ref{nrTheoremBLPBySimpleGroups1a} to \ref{nrTheoremBLPBySimpleGroups1e} is shown in \cite[Theorem 2]{BLP} using results from \cite{FV1998} and \cite{DalmauPearson}.
\end{proof}

In addition to these equivalent conditions, loop conditions of permutation groups are an easy and common way to show that a structure cannot primitively positively construct another. This paper discusses which of the conditions imply each other which can lead to a more efficient check of these conditions. See Section~\ref{SectionExamplesConditions} for a discussion of common conditions and their equivalent characterizations.

\subsection{Open problems}
With a view toward classifying all structures up to primitive positive constructions, there are two canonical questions concerning permutation groups:

\begin{question}
    Given any finite structure $\structA$, decide wether it is primitively positively interconstrucable with a permutation group.
\end{question}
It is easy to see that such a structure needs to admit a Pixley operation, named after \cite{Pix63}, that is a polymorphism $p$ satisfying 
$$p(x,y,x)=p(y,y,x)=p(x,y,y)=p(x,x,x)$$
as every permutation group has such an operation. However this condition is not enough as the structure 
$$
    (\{0,1,-1\};\{1\},\{(x,y)\mid |x|=|y|\},\{(x,y)\mid x+y=0\} ),
$$
called $\overline{\mathcal{M}_1}$ in \cite{FKRV25}, has a Pixley operation but cannot be pp-constructed from a permutation group.

We have a conjectured answer to this question:
\begin{conjecture}
    Given any structure $\structA$. The following are equivalent:
    \begin{enumerate}
        \item The structure $\structA$ has a Maltsev operation i.e. 
        $$\exists f \in \Pol(\structA):\forall x,y: f(x,y,y)=f(x,x,x)=f(y,y,x)$$
        and it has a guarded fully symmetric operation of all arities i.e. for all $n$, we have
        \begin{align*}
        \exists f \in \Pol(\structA)&:
        \\
        \forall x,y_1,\dots,y_n, \forall \sigma\in S_n
        &:
        f(x;y_1,\dots,y_n)=f(x;y_{\sigma(1)},\dots,y_{\sigma(n)}),
        \\
        \forall x_1,x_2,y
        &:
        f(x_1;y,\dots,y)=f(x_2;y,\dots,y)
        \end{align*}
        \item The structure $\structA$ is pp-interconstructable with a permutation group or $P_1$, the structure $(\{0,1\};<)$.
    \end{enumerate}
\end{conjecture}

The other question takes any structure $\structA$ and looks at the generalization of the set $\ppchar(\structA)$. Given any structure $\structA$, define $\ppchar(\structA)$ to be the set of all finite groups $G$ with trivial Frattini subgroup such that $\structA$ pp-constructs $\structure(G\act \prim(G))$.
\begin{question}
     For which sets $S$ of groups is there a finite structure $\structA$ such that $\ppchar(\structA)=S$?
\end{question}
For every finite set, $\structA$ exists and can be chosen to be a permutation group by this results. For infinite sets, it is known by the cyclic terms theorem \cite[Theorem 4.2]{Cyclic} that $S$ has to contain all cyclic groups (in which case $\structA$ is not Taylor) or it contains only finitely many of them. We believe to be able to generalize the cyclic terms theorem to get the following classification:
\begin{conjecture}
    Let $S$ be a set of groups (considered up to isomorphism) with trivial Frattini subgroup. Then, the following are equivalent:
    \begin{enumerate}
        \item There exists a finite structure $\structA$ such that $\ppchar(\structA)=S$.
        \item The set $S$ is closed under epimorphisms. Moreover,
        \begin{itemize}
            \item $S$ contains all groups (with trivial Frattini subgroup) or
            \item there is a finite set of primes $P$ such that $S$ contains all $p$-groups for all $p\in P$ and only finitely many other elements.
        \end{itemize}
    \end{enumerate}
\end{conjecture}
This proof will be presented in an upcoming paper.

\appendix

\section{Examples: Conditions arising from common group actions.}
\label{SectionExamplesConditions}
We take a closer look at some examples arising from common group actions. These are in particular $S_n \act [n]$, the condition of admitting a fully symmetric polymorphism and $G\act G$, the condition related to the regular action.

\begin{example}[The symmetric group] \label{ExampleFS}
The symmetric group on $n$ elements is a common condition. The characteristic set $\condcharred(\Sym(n)\act [n])$ is given for $n=$ 
    \begin{enumerate}[label=\arabic*]
        \item by an empty set. So $\condition(\Sym(1)\act [1]) \iff \top$
        \item by $\{\IZ/2\IZ\}$. So $\condition(\Sym(2)\act [2]) \iff \condition(\IZ/2\IZ\act \IZ/2\IZ)$. This is in fact the same action.
        \item by $\{\IZ/3\IZ, \Sym(3)\}$. 
		\item by $\{\IZ/2\IZ, A_4\}$, as those two groups have actions on $[4]$ without fixed point and every other subgroup of $\Sym(4)$ that is not contained in a conjugate of $\Sym(3)$ has an epimorphism to $\IZ/2\IZ$.
        \item by $\{\Sym(5),A_5, F, D_{5}, \Sym(3), \IZ/6\IZ, \IZ/5\IZ\}$ as $\Sym(5)$ has up to conjugation 8 subgroups which act on $[5]$ without fixed point (or, equivalently, are no subgroup of a conjugate of $A_4$) of which only $D_{6}$ has an epimorphism to $S_3$ and is therefore not to mention here. 
        
        A simpler, but strictly stronger set would be $\{A_5, \IZ/2\IZ, \IZ/3\IZ, \IZ/5\IZ\}$ as every of the above groups has an epimorphism into one of these groups, so
        \begin{align*}
            \condition(\IZ/2\IZ\act \IZ/2\IZ)&\land \condition(\IZ/3\IZ\act \IZ/3\IZ) \\
            {} \land \condition(\IZ/5\IZ\act \IZ/5\IZ) &\land \condition(A_5 \act \prim(A_5))
            \implies \condition(\Sym(5)\act [5])
        \end{align*}
        but this is not an equivalence. This is the first $n$ where a group appears that has no epimorphism into a cyclic group.
    \end{enumerate}
    In general, we get that $\{\condition(G\act \prim(G))\mid G \text{ subquotient of }S_n\}$ elementary implies $\condition(S_n\act [n])$ as every nontrivial subquotient of $S_n$ has an epimorphism into a simple subquotient of $S_n$
\end{example}

\begin{example}[The regular representation]
    For any group $H$, the condition $\condition(H\act H)$ induced by the group acting on itself by left multiplication, is elementary equivalent to 
    $$
        \bigwedge_{\{1\} \lneq S \le H} \condition(S \act \prim(S)),
    $$
    as every subgroup of $H$ except for the trivial group has no fixed point in this action. Note that some of the conditions might be superfluous. To omit them, we consider solvable groups separately.
    \begin{itemize}
        \item 
        For a solvable group $H$ with $|H|=n$, we get that $\condition(H\act H)$ is elementary equivalent to 
        $$
            \bigwedge_{p|n \text{ prime}} \condition(\IZ/p\IZ \act \IZ/p\IZ),
        $$
        as every nontrivial subgroup of $G$ has no fixed point. Moreover, by Sylow's Theorem, all of those goups are actually isomorphic to a subgroup of $H$ and these are the only cyclic subgroups up to isomorphism. As $H$ is solvable, every subgroup of $H$ admits an epimorphism into a cyclic subgroup of $H$. Moreover,
        $$
            \condcharred(H\act H)=\{\IZ/p\IZ \mid p|n \text{ prime}\}
        $$
        as none of the conditions is superfluous in this case.
        \item 
        For a non-solvable group, the condition $\condition(H\act H)$ can never be described by the cyclic subgroups, as there is always a subgroup $S\le H$ such that $S$ has no epimorphism into an abelian simple group. The group $S$ can be constructed as the smallest subgroup of $H$ that is contained in a composition series of $H$ and admits a non-cyclic simple group as composition factor.\qedhere
    \end{itemize}
\end{example}

%% Weitere Resultate für potentielle nachfolgende Veröffentlichungen
%\clearpage
%\input{06_a_Universal-Preliminaries}
%\input{06_b_UniversalStructures}
%\input{07_P-groups}
%\input{08_CyclicTerms}
%\input{09_GroupTheory}
%\input{10_The-trivial-group}
%\input{99_todos}

\bibliography{citations} 
\bibliographystyle{alpha}

\end{document}